\documentclass[reqno,11pt]{amsart}

\usepackage[margin=1in]{geometry}

\usepackage{cite}
\usepackage{amsmath}
\usepackage{stmaryrd}
\usepackage[all]{xy}
\usepackage{amsfonts}
\usepackage{mathrsfs}
\usepackage{amssymb}
\usepackage{color}

\usepackage{graphicx,color,wrapfig}
\usepackage{tikz}

\newtheorem{thm}{Theorem}[section]
\newtheorem{lem}[thm]{Lemma}
\newtheorem{prop}[thm]{Proposition}
\newtheorem{cor}[thm]{Corollary}

\newtheorem{df}[thm]{Definition}
\newtheorem{qu}[thm]{Question}
\newtheorem{rmk}[thm]{Remark}
\newtheorem{nota}[thm]{Notation}
\newtheorem{ex}[thm]{Example}
\newtheorem{caveat}[thm]{Caveat}
\newcommand{\cyccirc}[2]{\leftsub{#1}{\bar{\circ}}_{#2}}

\newcommand{\leftsub}[2]{{\vphantom{#2}}_{#1}{#2}}

\def\nn{\nonumber}

\def\eps{\epsilon}
\def\G{\Gamma}
\def\arity{ar}
\def\gVect{g\mathcal{V}ect}
\def\dgVect{dg\mathcal{V}ect}
\def\t{\tau}

\def\tensor{\otimes}

\def\la{\langle}
\def\ra{\rangle}

\def\CalC{{\mathcal C}}

\def\CO{{\mathcal O}}

\def\F{\mathcal F}

\def\R{{\mathbb R}}

\def\Z{{\mathbb Z}}
\def\SS{{\mathbb S}}
\def\N{{\mathbb N}}
\def\G{\Gamma}
\def\V{{\mathcal V}}
\def\Iso{{\mathcal I}so}
\def\RT{{\mathcal RT}}

\def\T{{\mathcal T}}

\def\G{{\mathcal Graph}}

\def\del{\partial}
\def\colim{\mathrm colim}

\newcommand{\op}{\mathcal}

\def\O{{\mathcal O}}
\def\P{{\mathcal P}}
\def\Sn{{\mathbb S}_n}
\def\SS{{\mathbb S}}

\def\Snp{\SS_{n+}}
\def\Endo{{\mathcal E}nd}
\def\cEndo{\mathcal E}

\def\Op{\O^{\oplus}}

\def\OSp{\O^{\oplus}_{\SS}}
\def\OSpp{\O^{\oplus}_{\SS_+}}

\def\form{\la \; ,\; \ra}
\def\Gbracket{\{\;\bullet\;\}}
\def\Liebracket{[\;\circ\;]}
\def\cLiebracket{[\;\odot\;]}
\def\cGbracket{\{\;\odot\;\}}
\def\odo{\otimes \cdots \otimes}
\def\K{\mathfrak K}
\def\crl{*}

\def\ast{}
\def\Fin{Fin}

\def\Mgn{\overline{M}_{g,n}}
\def\Mksv{\Mgn^{KSV}}

\newcommand\ccirc[2]{\, \leftsub{#1}{\circ}_{#2}}
\def\scirct{\ccirc{s}{t}}

\newcommand\cbullet[2]{\, \leftsub{#1}{\bullet}_{#2}}

\newcommand\Sccirc[2]{\, \leftsub{#1}{\circ}_{#2}^{S^1}}

\def\Arc{\mathcal {A}rc}

\newcommand{\ds}{\displaystyle}

\def\T{{\mathbb T}}
\def\D{\mathfrak D}

\def\unit{1\!\!1}

\def\s{\mathfrak s}
\def\sout{s_{\rm out}}
\def\sin{s_{\rm in}}

\newcommand{\oprod}{\op{O}^{_{\prod}}_{\SS}}
\newcommand{\xycirc}[2]{\leftsub{#1}{\circ}_{#2}}

\newcommand{\fr}{\mathfrak}

\begin{document}

\title{The odd origin of Gerstenhaber brackets, Batalin-Vilkovisky operators, and  master equations}

\author
{Ralph M.\ Kaufmann}

\address{Purdue University Department of Mathematics,
 West Lafayette., IN 47907}
 \email{rkaufman@math.purdue.edu}

\author{ Benjamin C.~Ward}
 
 \address{ Simons Center for Geometry and Physics, Stony Brook, NY 11794}
 \email{bward@scgp.stonybrook.edu}
\author{J.~Javier Z\'u\~niga }

 \address{Departamento Acad\'{e}mico de Econom\'{\i}a
Universidad del Pac\'{\i}fico, Lima 21,
Per\'u}
\email{zuniga\_jj@up.edu.pe}

\begin{abstract}
Using five basic principles we treat Gerstenhaber/Lie brackets, BV operators and Master equations appearing in mathematical and physical contexts in a unified way.
The different contexts for this are given by
the different types of (Feynman) graphs that underlie the particular situation.

Two of the maxims we  bring forth are (1) that  extending
to the non-connected graphs gives a commutative multiplication forming
a part of the BV structure and (2) that there is a universal odd twist
that unifies and explains seemingly ad hoc choices of signs, and is responsible for
the BV operator being a differential.

Our treatment results in uniform, general theorems. These allow us to prove new results and
recover and connect
many constructions that have appeared independently throughout
the literature. The more general point of view also allows us to disentangle the necessary from the circumstantial.
\end{abstract}
\maketitle

%\tableofcontents

\section*{Introduction}

In recent years, there have been many algebraic constructions which in their
background have some physical, (string) field theoretical origin. Perhaps the most prominent are Lie brackets,
Gerstenhaber brackets and master equations. The Lie algebras of Kontsevich \cite{Kthree,CV}
as well as Deligne's conjecture  \cite{KS,McSm,Vor,del,BF,T}, its cyclic generalization \cite{cyclic}
and its $A_{\infty}$ version which was
studied in \cite{TZ,KSchw,manfest,Ward}, and notably
string topology \cite{CS} are of this type, especially when considered in the algebraic framework
\cite{del,cyclic, TZ}.  Among master equations the relevant constructions go back to Sen and Zwiebach \cite{Zwie,KSV} and newer
ones include \cite{KontSchw,Schwarz, HVZ,S1,S2,S3,Ba,MMS,Merkulov}.
There is a plethora of further incidences which would fill volumes. One particularly important aspect for us is that
solutions give rise to a viable action as explained
in \cite{Zwie,KontSchw,Costello,Merkulov}.
Without being too specific in this introduction there are several incarnations of the master equation going by various names:
\begin{equation}
\{S\bullet S\}=0, \quad dS +\frac {1}{2}\{S\bullet S\}=0, \quad dS+\Delta(S)+\frac{1}{2}\{S\bullet S\}=0.
\end{equation}
The first is a type of classical master equation (ME), with the differential $d$ the equation is sometimes called the Maurer-Cartan (MC) equation and
with $\Delta$ is called the quantum master equation (QME). Of course, one has to ---and we will--- specify where $S$ lies and what the definition of $\{\, \cdot \, \}$,
$\Delta$ is.
The physical setup using correlation functions for graphs,  is mathematically captured by an operadic type context.  The operadic context is fixed by the type of (Feynman) graphs one allows, e.g. with or without loops; see e.g.\ Table \ref{optable}.
We show how the algebraic operations as well as
master equations both appear naturally under the same principles in all situations.

The paper provides  both  an introduction to the theory and the classic context
as well as gives new, state-of-the-art results in an accessible
fashion and is intended for a general readership.
To this end, we will go through the different cases  by starting with the most familiar, operads,
and make our way to less widely known subjects such as modular operads
and $\K$ twists through a progressive development.
This yields a systematic study of the above mentioned algebraic operations, i.e.\ brackets
and BV operators $\Delta$,
their occurrence in master equations, and the origin of these equations. The contexts we treat are
operads, cyclic operads, dioperads, (wheeled) PROP(erad)s, (wheeled)
PROPs and modular operads as well as their non--connected (nc) versions,
nc--(modular/cyclic/di--) operads,  which we newly define in this paper. All of them are
defined in the paper to make it self--contained.
Along the way, we provide new results and constructions. These include the bracket in the cyclic case
and several non--connected versions of the above structures, which are necessary to construct BV operators.

Besides these new results another key point of the paper is to establish
that there is a
universal framework of odd (more technically $\K$--twisted) structures, regardless
of the details of the specific example which naturally explains all the constructions, including degrees and signs,
in one fell swoop ---instead of handling each of the case by case with possibly different but equivalent conventions.
This allows us to disentangle
different structures that often appear in conjunction in specific examples. For instance, brackets
on homomorphism spaces and brackets involving symplectic constructions,
which {\em inherit} their structures from being {\em examples} of $\K$--twisted objects.
A second main point is that we show that for each framework there is a non--connected analog, which
has an additional ``horizontal'' multiplication. Physically this corresponds to going from connected to non--connected Feynman graphs.
The framework itself can even be generalized further in terms of categorical language. This is done in \cite{feynman}.
Here we give the concrete constructions of that general theory in the context relevant for the physical and geometric examples listed above. Avoiding the categorical complications to make our statements,
let us summarize the results from \cite{feynman} as they apply to the current situation into the following mantra:
\begin{enumerate}
\item Odd non--self--gluings give rise to odd Lie brackets.
\item Odd self--gluings give rise to differentials.
\item The horizontal multiplication   turns the odd brackets into odd
Poisson or Gerstenhaber brackets and makes the differentials BV operators (on the nose and not just
up to homotopy).
\item Algebraically, the master equation classifies dg--algebras over the relevant dual or Feynman transform.
\item Topologically, the master equation drives the compactification.
\end{enumerate}
Here we prove that in all the examples above this mantra turns into algebraic and topological theorems after we provide the correct structures, which are constructed in this paper.

For the transforms we note that we only consider transforms in which the elementary operations are one--edge gluings on the underlying graphs and do not include the horizontal compositions into the differential. This is opposed to the discussion in \cite{Val,Merkulov}.
 Physically the reason for this is that the underlying propagators of single particles  are fundamental and one
 wants to preserve the fact that the exponential of
the action gives the sum over all non--connected graphs and hence
the equivalence
\begin{equation}
dS+\Delta(S)+\frac{1}{2}\{S\bullet S\}=0 \Leftrightarrow  (d+\Delta) e^S=0
\end{equation}
 This guarantees that we do not
change the fundamental exponential/log
 relation between the generating functions
for the connected
graphs and the non--connected ones.

We define each notion from scratch and along the way   treat the complicating issues of  signs,
 the marked difference between the directed case, e.g.\ operads
and the non-directed case, e.g.\ cyclic operads,
 and  the difference between the symmetric and non--symmetric versions.
We also show that in particular cases the odd structures can be shifted back to even ones,
if there are spurious isomorphisms of twists.

Finally, we firmly anchor our point of view in  geometry as suggested by  open/closed string field theory.
Here we consider as a main source
of examples operadic--type structure with compatible $S^1$ actions. This leads to one--parameter
family gluings. On the chain level, these are naturally odd, since they have degree $1$.
This explains all the signs, as we have postulated. We  also examine different geometrical
origins such as framings in the so--called open case. We then argue that in these situations
mantra (5) appears as the natural point of view.
Finally, we summarize  our results, put them into context of the existing literature and give an outlook.

The paper is organized as follows:
In \S1 we start by recalling the classical results for operads and Gerstenhaber brackets, which we generalize and introduce
our general view of signs and degrees. In \S2 we provide a new generalization of the bracket to the cyclic case, which we discuss in detail. As examples we recover results of \cite{Kthree,CV,MenLie}. \S3 deals with the notions that have input/output distinction of the graphs and we prove all of the points of the mantra for them using the transforms indicated above.
\S4 gives the background theory for the modular case, that is self--gluings and no input/output distinction, in detail.
This is the technically most challenging, but it is here that we can introduce the general $\K$-twist as the definitive version of odd. We show that all conventions we have made up to that point for signs agree with this general $\K$--twist.
\S5 then gives the BV operator for this case. \S6 contains the new constructions of non--connected operads and their cyclic and modular versions in order to fulfill mantra (3). Here we also newly introduce the operator $B_+$ of renormalization  into the operadic picture. \S7 contains the the definition of the transforms for mantra (4) and its incarnation into theorems.  \S8 provides geometric background for odd operations. Here we newly treat the $\Arc$ operad and hence string topology.
\S9 gives a summary, discussion and outlook. We include an Appendix on graphs and algebras for the readers convenience.

\section{Operads, Gerstenhaber's bracket and  natural ``oddness''}\label{opsec}

\subsection{Basic Background}

For convenience, we usually work in the in the category
$\gVect$ of  graded vector spaces
over a fixed ground field $k$ of characteristic $0$.

For most constructions, this is not necessary and one can generalize to any additive category (or better a category enriched over graded Abelian groups) which is cocomplete.
Or even less, where the particular colimits we use exist.

Sometimes we however use the isomorphism between $\Sn$ invariants and $\Sn$ co--invariants for all $n$. In this
case, we need characteristic $0$. Usually this step is again convenient but not strictly necessary and it can
be omitted at the price of less succinct statements.

\subsubsection{Canonical Example} For a finite dimensional vector space $V$, define a vector space $\Endo(V)(n):=Hom(V^{\otimes n},V)$.  These spaces have an obvious $\Sn$ action by permuting the variables (factors of $V$) of the multilinear functions.
There are composition maps $\circ_i$: $\Endo(V)(n)\otimes\Endo(V)(m)\to \Endo(V)(n+m-1); f\otimes g \mapsto f\circ_i g$
 which are given by substituting  $g$ in the $i$--th place of the function $f$. There is a unit
 for these compositions which is the identity function $id:V\to V$.
These compositions are associative and equivariant under the action
of the relevant symmetric groups in a natural universal
 manner. That is for every pair of permutations $(\sigma,\sigma')\in \Sn\times
\SS_m$ there is a unique permutation $\sigma\circ_i\sigma'\in \SS_{n+m-1}$
s.t.\ $\sigma f\circ_{i} \sigma' g= (\sigma \circ_i\sigma') f\circ_{\sigma^{-1}(i)}g.$

\subsubsection{Operads} Making the example above abstract:  An operad is given by a collection $\{\O(n)\}$ in $\gVect$ or more generally in a symmetric monoidal category $\CalC$ together with:

\begin{enumerate}
\item operadic compositions or gluing maps,
$$
\circ_i:\O(n)\otimes \O(m)\to \O(m+n-1): 1\leq i\leq n
$$
\item an $\Sn$ action for each $\O(n)$. 
\item and a unit $id\in \O(1)$.
\end{enumerate}
Such that the gluing maps satisfy the  associativity relations
\begin{multline}\label{opassoceq}
(\op{O}(n) \circ_i \op{O}(m)) \circ_j \op{O}(l) =\\
\left\{ \begin{array}{rl} \gamma_{\op{O}(m),\op{O}(l)}(\op{O}(n) \circ_j \op{O}(l)) \circ_{i+l-1} \op{O}(m) & \mbox{if   } 1\leq j<i \\ \op{O}(n) \circ_i (\op{O}(m) \circ_{j-i+1} \op{O}(l)) & \mbox{if   }  i\leq j \leq i+m-1 \\ \gamma_{\op{O}(m),\op{O}(l)}(\op{O}(n) \circ_{j-m+1} \op{O}(l)) \circ_i \op{O}(m) & \text{if } i+m \leq j \end{array} \right.
\end{multline}
where $\gamma$ is the braiding isomorphism in the symmetric monoidal category.  In the category $\gVect$, $\gamma(a\tensor b)=(-1)^{deg(a)deg(b)}b\tensor a$, where $deg$ is the degree.

The unit satisfies
$$\forall a\in \O(n), 1\leq i\leq n:\quad id\circ_1a=a; \quad a\circ_iid=a $$
and the gluing maps are required to be $\Sn$ equivariant.  We omit the rather lengthy formal definition of the equivariance, as it can easily be derived from the canonical example above.  Details can be found in \cite{MSS}.

A collection of $\{\O(n)\}$ of $\Sn$ modules $\O(n)$ is called an $\SS$--module.

\subsubsection{Rooted trees} The associativity means that any planar rooted
tree $\tau$ with leaves labeled by $1,\dots, n$
determines a unique composition by using it as a flow chart. Here
associativity says that the order of the compositions is irrelevant.
If we add the $\SS$ equivariance, then any rooted tree gives an operation.
More in \S\ref{triplepar}.

\subsubsection{Algebras over operads}
The operad $\Endo(V)$ plays a special role.
An algebra $V$ over an operad is an operadic morphism from $\O$ to $\Endo(V)$ (of degree $0$). Here operadic morphism is the straightforward notion
obtained by requiring that all the compositions and $\Sn$ actions are respected.

The operad $\Endo(V)$
can also be generalized to any closed symmetric monoidal category $\CalC$ where
now $V$ is an object.

\subsubsection{Weaker structures}
Dropping the unit from the data and axioms yield the notion of a non-unital pseudo-operad.   Dropping the $\Sn$ action and the $\Sn$ equivariance, we arrive at the definition of a {\em non-$\Sigma$} operad.

The distinction between pseudo or not is irrelevant in the unital case as these notions are equivalent; see \cite{MSS}.

\subsection{Lie bracket}

\begin{nota}  Given an operad $\op{O}=\{\op{O}(n)\}$ we set $\Op:=\bigoplus_{n\in \N}\O(n)$.  If $a\in \O(n)$ with degree $deg(a)$,
we set $\arity(a)=n$ and $|a|:=deg(a)+\arity(a)$.  These two gradings on $\Op$ are refered to respectively as the internal degree and the total degree.  We will also consider the co-invariants $\O(n)_{\Sn}$ and set $\OSp:=\bigoplus_{n\in \N}\O(n)_{\Sn}$.
\end{nota}

\begin{thm}\cite{GV,KM}
Given an operad or non-$\Sigma$ operad $\op{O}=\{\O(n)\}$ in $\gVect$, set
\begin{equation*}a\circ b:=\sum_{i=1}^{\arity(a)} a\circ_i b
\end{equation*}
then  $\circ$ is a pre-Lie multiplication and hence
\begin{equation*}
[a \circ b]:=a\circ b - (-1)^{deg(a)deg(b)}b\circ a
\end{equation*}
defines a graded Lie bracket on $\Op$ with respect to the internal grading.  In the symmetric case this Lie bracket descends to a Lie bracket on $\OSp$.
\end{thm}

\subsection{Odd Lie bracket}
In Gerstenhaber's original work \cite{Gerst} the bracket is not Lie but odd Lie and certain signs are introduced in the summation.  We will show that these signs can be understood in terms of operadic suspensions and shifts.  In particular, doing an operadic suspension one almost gets the signs.  After one more shift, the signs are the ones of the Hochschild complex. What seems {\it prima vista} unfortunate, namely that a na\"ive shift of an operad ceases to be an operad, is actually completely natural, as according to the mantra the bracket should come from an odd gluing.
Let us formalize this.

\subsubsection{Shifts and odd Lie brackets}
Given a graded vector space $V=\bigoplus_i V^i$,
we set $\Sigma V:=V[-1]$ this means that $(\Sigma V)^i=V^{i-1}$ and call it the suspension
of $V$.  The inverse operation of suspension is called desuspension. We set $(\Sigma^{-1} V)^i=V^{i+1}$

If $|\, .\,|$ is the grading of $V$, we set $s(a):=|a|-1$
then $s(a)$ is the natural degree of $a$ thought of as an element in $\Sigma V$.

Recall (e.g. from the appendix) that an odd Lie bracket is a bilinear map (by our conventions of degree $-1$) which satisfies odd anti-symmetry and odd Jacobi identities.  Alternatively, a direct calculation yields the following useful
characterization.
\begin{lem}
\label{shiftlem}
$\{\; \bullet \;\}$ is an odd Lie bracket on $V$ if and only if it is a Lie bracket on $\Sigma^{-1} V$.
\end{lem}
\begin{rmk}
Since we are dealing with signs only, the shift in degree can be made to be $+1$ or $-1$.
\end{rmk}

\subsubsection{Shifted compositions and  Gerstenhaber's bracket}
Following Gerstenhaber \cite{Gerst}, given $\O$ in $\op{V}ect$ we define new composition maps $\bullet_i$ as $a\bullet_i b := (-1)^{(i-1)s(b)} a\circ_i b$.  More generally if $\O$ is an operad in $\gVect$ we define
\begin{equation*}
a\bullet_i b := (-1)^{(i-1)(ar(b)-1)+(ar(a)-1)deg(b)} a\circ_i b
\end{equation*}
Set
\begin{equation*}
a \bullet b=\sum_{i=1}^{\arity(a)} a\bullet_i b
\end{equation*}

Analogously to the Lie situation, set
\begin{equation*}
\{a\bullet b\}:= a\bullet b - (-1)^{s(a)s(b)}b\bullet a
\end{equation*}

With this definition one readily verifies that:
\begin{prop}\label{Gbracketprop}\cite{Gerst,GV,KM,del} The bilinear operation $\{\;\bullet\;\}$ is an odd Lie bracket on $\Op$, with respect to the total grading, and it descends to co-invariants $\OSp$.
\end{prop}

\subsection{Suspensions and Shifts for Operads}  Observe that as a graded $\Sn$ representation,
\begin{equation*}
\Endo(\Sigma^{-1}k)(n)\cong \Sigma^{n-1}sgn_n,
\end{equation*}
where $sgn_n$ is the sign representation of $\Sn$.  Also, note there is a na\"ive product of operads defined as $(\O\otimes \P)(n):=\O(n)\otimes \P(n)$ with the diagonal $\Sn$ action and compositions.
\begin{df}
Given an operad $\O$ we define $\fr{s}\O$, the operadic suspension of $\O$, to be the operad $\op{O}\tensor\Endo(\Sigma^{-1}k)$.  We define $\fr{s}^{-1}\O$, the operadic desuspension of $\O$, to be the operad $\op{O}\tensor\Endo(\Sigma k)$.
\end{df}

Explicitly, $\fr{s}\O(n)=\Sigma^{n-1}(\O(n)\otimes sgn_n)$ with the natural induced diagonal operad structure.  Identifying elements of $\op{O}$ with their counterparts in $\fr{s}\op{O}$, we have:

\begin{prop}
Gerstenhaber's bracket $\Gbracket$ (Proposition $\ref{Gbracketprop}$) agrees with the natural Lie bracket $\Liebracket$ associated to the suspended operad $\fr{s}\O$.
\end{prop}
\begin{proof}
First observe that the total grading in $\op{O}$ differs with the internal grading of $\fr{s}\op{O}$ by a na\"ive shift.  Thus the above identification takes a degree $-1$ operation to a degree $0$ operation.

Now, using the Koszul sign rule we see that in the operad $\Endo(\Sigma^{-1} k)$ the operad composition operation is given by:
\begin{equation} \label{suspeq}
\circ_i\colon(e_1\wedge\dots\wedge e_n) \tensor (e_1\wedge\dots\wedge e_m)  \mapsto  (-1)^{(i-1)(m-1)} (e_1\wedge\dots\wedge e_{n+m-1})
\end{equation}
and we want to show that under this identification, the structure maps $\widetilde{\circ}_i$ of $\fr{s}\op{O}$ satisfy
\begin{equation*}
a\tilde{\circ}_ib = (-1)^{(i-1)(m-1)+(n-1)\text{deg}(b)} a\circ_ib
\end{equation*}
for $a\in \op{O}(n)$ and $b\in\op{O}(m)$.  To see this consider the sequence:
\begin{eqnarray*}
\fr{s}\op{O}(n)\tensor \fr{s}\op{O}(m)&=& \op{O}(n)\tensor \Sigma^{n-1}sgn_n\tensor \op{O}(m)\tensor \Sigma^{m-1}sgn_m  \\ &\to& \op{O}(n)\tensor \op{O}(m)\tensor \Sigma^{n-1}sgn_n\tensor \Sigma^{m-1}sgn_m \stackrel{\circ_i\tensor\circ_i}\longrightarrow \op{O}(n+m-1)\tensor \Sigma^{n+m-2}sgn_{n+m-1}
\end{eqnarray*}
applying the symmetric structure in the first arrow gives $(-1)^{(n-1)\text{deg}(b)}$ and applying the diagonal $\circ_i$ map gives $(-1)^{(i-1)(m-1)}$ as per line $\ref{suspeq}$.  Thus $\tilde{\circ}_i=\bullet_i$.
\end{proof}

\begin{cor}  The operations $\bullet_i$ satisfy the following `odd' associativity relations ---compare to equation $\ref{opassoceq}$.
\begin{equation}
\label{oddassoc}
(a \bullet_i b) \bullet_j c =
\begin{cases} (-1)^{(|b|-1)(|c|-1)}(a \bullet_j c) \bullet_{i+l-1} b & \mbox{if   } 1\leq j<i \\ a \bullet_i (b \bullet_{j-i+1} c) & \mbox{if   }  i\leq j \leq i+m-1 \\ (-1)^{(|b|-1)(|c|-1)}(a \bullet_{j-m+1} c) \bullet_i b & \text{if } i+m \leq j
\end{cases}
\end{equation}
\end{cor}

\subsubsection{Signs: An Essential Remark}
There are two ways in which to view the signs associated to the odd operations.
\begin{enumerate}
\item Simply as the shifted signs which may seem rather odd.
\item By setting $deg(\bullet)=1$ and using the Koszul rule of sign when permuting symbols.Here the symbols ``$\{$'' and ``$\}$'' are assigned degree $0$. That is as a $\Z/2\Z$ graded operation $\bullet$ is odd.
\end{enumerate}
In  geometric considerations, the operation $\bullet$ indeed often comes from an $S^1$ action, which one can consider the $\bullet$ to represent.

\begin{rmk}
In the operad or cyclic operad case (see the next section) the first version is viable, while in the modular (see section \ref{triplepar})
or more general case the second version is preferable and in fact necessary. Thus with hindsight, we will see that the second version is actually natural also in the non-modular context.
\end{rmk}

\subsubsection{Motivational example for $\fr{s}\O$}
Considering the endomorphism operad $\Endo(V)$, the operadic suspension arises if one considers $V[1]$ instead of $V$.
A map of degree $0$ from $V^{\otimes n}\to V$ gives a map of degree $n-1$ from $(V[1])^{\otimes n}=V^{\otimes n}[n]\to V[1]$ and one may show that as operads $\Endo(V[1])\simeq \fr{s}\Endo(V)$ (see e.g.~\cite{MSS}). Consequently,

\begin{prop}\cite{MSS}
\label{shiftprop}
$V$ is an $\O$-algebra if and only if $V[1]$ is an $\fr{s}\O$-algebra.
\end{prop}

\subsubsection{Degrees in the Hochschild complex}
If $A$ is an associative algebra, the spaces $\Endo(A)(n)$ form a complex,
the Hochschild cochain complex $CH^*(A,A)$. It is given by $CH^n(A,A)=Hom(A^{\otimes n},A)$
with the Hochschild differential, which is immaterial at the moment. To be a complex, $CH^*(A,A)$ must take the total grading, but this is not the natural operadic grading which is either $deg(a)$ in $\Endo(A)$ or $s(a)$ in $\fr{s}\Endo(A)=\Endo(A[1])$.

So although the operadic suspension $\fr{s}\Endo(V)$ of $\Endo(V)$ is a graded operad and it provides
Gerstenhaber's signs as the signs of the natural Lie bracket, as a graded vector space it is still one shift short from the Hochschild complex.  Adding one more na\"ive shift $\Sigma$, we obtain the right grading, so that $CH^*(A,A)$ is a dg algebra with respect to the cup product and the bracket has Gerstenhaber's signs, that is $CH^*(A,A)=\Sigma \fr{s}\Endo(A)$.  Presently we formalize this by introducing the notion of an odd operad.

\subsection{Na\"ive shifts and odd operads}\label{oddsec}

\begin{df}
For an $\SS$--module $\O$ its suspension $\Sigma \O$ is the $\SS$--module
$\{\Sigma \O(n)\}$. Likewise we define $\Sigma^{-1}\O$.
\end{df}

\begin{df}
An odd operad $\O$ in $\gVect$ is an $\SS$--module with operations $\bullet_i$ such that $\Sigma^{-1}\O$
together with the $\bullet_i$ is an operad.
\end{df}

Notice this means that in $\O$ the operations satisfy the equations
\ref{oddassoc} where $|\,.\,|$ is now just the degree in $\O$.

\begin{prop}
Given an odd operad $\O$, the vector space $\Op$ carries
an odd bracket $\Gbracket$.
\end{prop}
\begin{proof}
This follows directly from Lemma \ref{shiftlem}.
\end{proof}

\begin{cor} \label{r1}
 Given an operad $\CO$
the odd operad $\Sigma \fr{s}\CO$ naturally carries an odd Lie bracket,
which is the shift of the natural Lie bracket on $\fr{s}\CO$.
\end{cor}

\subsubsection{The Hochschild complex as an odd operad}
To sum up this section, the most natural way to think about the Hochschild complex is as an odd operad
$CH^*(A,A)=\Sigma \fr{s} \Endo(A)$. This provides all the correct signs and degrees.  In this fashion one can generalize the bracket to the
cyclic and modular cases.

We briefly collect together the relevant degrees in Table \ref{degtable}.

\begin{table}
\begin{tabular}{l|l}
$a\in$&natural degree of $a$\\
\hline
$\CO(n)$&$deg(a)$\\
$s\CO(n)$&$s(a)=deg(a)+n-1$\\
$\Sigma s\CO(n)$&$|a|=deg(a)+n$
\end{tabular}
\caption{\label{degtable} Natural degrees in suspensions and shifts}
\end{table}
%
%\subsection{Tree picture and twisted Operads}
%

%One way to think about odd operads is that in the normal picture of
%operads the trees have
%been replaced by rooted trees whose internal edges and root edge
%each have weight one. This is the same as giving the vertices weight one
%as in a rooted tree every vertex has a unique outgoing edge.
%
%

\section{Cyclic, anti-cyclic operads and a cyclic bracket}\label{cycsec}
The first generalization we will give is for the cyclic case.
We briefly recall the definitions in terms of operads with extra structure and in terms of arbitrary finite sets.

\subsection{The $\Snp$ definition of cyclic operads}
In an operad one can think of $\O(n)$ as having $n$ inputs and one output. The $\Sn$ action
then permutes the inputs.
The idea of a cyclic operad is that the output is also treated democratically, i.e.\ there is an action of $\SS_{n+1}$ on $\O(n)$ which also permutes the output. Usually one labels the inputs by $\{1,\dots,n\}$
and the output by $0$. In order to formalize this we follow \cite{GeK1}
and define $\Snp$ to be the bijections of the
set $\{0,1,\dots,n\}$. Then $\Sn$ is naturally included into $\Snp$ as the bijections that keep $0$ fixed.
As a group $\Snp\simeq \SS_{n+1}$ and it is generated by $\Sn$ and the long cycle $\tau=(01234\cdots n)$. Let $C_{n+}\subset \Snp$ be the cyclic group generated by $\tau$.

Given an $\Snp$ module $(M,\rho)$ we denote the action of $\tau$ by $T$, i.e.\ for $m\in M$.
$T(m)=\rho(\tau)(m)$. We also define the operator
$N=1+T+\cdots +T^{n}$ on $\O(n)$.

\begin{df}\cite{GeK1}
A {\em cyclic operad} is an operad $\O$ along with an action of $\Snp$ action on each $\O(n)$
which extends the action of $\Sn$ such that the following conditions are met
\begin{enumerate}
\item $T(id)=+id$ where $id\in \O(1)$ is the operadic unit.
\item $T(a \circ_{\arity(b)} b)=+(-1)^{|a||b|}T(b)\circ_1T(a)$
\item  $T(a\circ_i b) = T(a)\circ_{i+1} b \text{ for } 1\leq i \leq n-1$.
\end{enumerate}

Alternatively, if the axioms hold after replacing the two $+$ signs with $-$ signs, then $\O$ is called an {\em anti-cyclic operad}.
A weaker structure than that of cyclic operad is that of a {\em non-$\Sigma$ cyclic operad}. Here one only requires
an action of $C_{n+}$, the cyclic subgroup of $\Snp$,  on $\O(n)$ along with the above axioms.
\end{df}

The collection of objects $\O(n)$ together with their $\Snp$ action is called a cyclic $\SS$-module.
In order to get the same indexing for the symmetric groups and the operad one sets $\O((n)):=\O(n-1)$.
Here, morally, $n$ is the number of inputs and outputs.

\begin{ex}
The standard example of a cyclic operad is $\Endo(V)$ where $V$ is a (graded) vector space of finite type with a (graded) non--degenerate even
bilinear form $\form$.
The operation $T$ on $f\in \Endo(n)$ is then defined via $\form$ by
\begin{equation}
\la v_0, Tf(v_1 \odo v_n)\ra= \pm \la v_n, f(v_0\odo v_{n-1})\ra
\end{equation}
where in the graded case $\pm$ is the sign given by the Koszul sign rules.
Another way to phrase this is as follows. $\form$ gives an isomorphism
between $V$ and its dual space $\check V$.
Thus
\begin{equation}
\Endo(V)(n)= Hom(V^{\otimes n},V)\simeq \check V^{\otimes n}\otimes V
\stackrel{\form}{\longrightarrow} \check V^{\otimes n+1}
\end{equation}
Now on the last term there is an obvious $\Snp$ action permuting the factors
and this action can be transferred to $\Endo(V)(n)$ via the isomorphism.
\end{ex}

\begin{ex}  Let $V$ be a symplectic vector space.  Then $\Endo(V)$ is an anticyclic operad.  The action is then given as in the last example. The extra minus sign comes from the fact that the symplectic form is skew symmetric.
\end{ex}

\begin{rmk}
The last two examples can be unified using the notion of operadic
correlation functions from \cite{hoch2}. Here the correlation
functions are given on $\check {V}^{\otimes n}$ and the
propagators by the Casimir elements of $\form$,
where now these elements encode the signs. This fits well with the tree
picture and Feynman diagrams since the propagators are associated to the {\em edges} and
not the vertices.
\end{rmk}

\begin{ex}
\label{cvex}
Examples of cyclic operads are given by the cyclic extension of the operads
$Comm$, $Lie$ and $Assoc$. These are the operads whose algebras are precisely
associative and commutative, Lie and associative algebras.
\end{ex}

\subsubsection{Algebras over (anti)-cyclic operads}
An algebra over a cyclic respectively anti-cyclic operad $\O$
is a vector space $V$ together with a non-degenerate even symmetric form
or respectively a non-degenerate even skew symmetric form and
a morphism of cyclic, respectively anti-cyclic operads from $\O$ to $\Endo(V)$.

\subsection{Products and suspension for (anti)-cyclic operads}
\label{productssec}

The product of two cyclic operads or two anti-cyclic operads is a cyclic operad while the product of a cyclic and an anti-cyclic operad is anti-cyclic.

\begin{ex}
Given an cyclic operad $\O$ and a symplectic vector space $V$ with a symmetric
non-degenerate pairing the operad $\O\otimes \Endo(V)$ is still anti-cyclic.
\end{ex}
Recall that the operadic suspension was defined via $-\tensor \Endo(\Sigma^{-1}k)$.  Since $\Sigma^{-1}k$ is canonically a symplectic vector space, we have:

\begin{lem}
The operadic suspension of a cyclic operad is an anti-cyclic
operad and vice-versa.
\end{lem}

\begin{ex}
In the case of $\Endo(V)$ for a pair $(V,\form)$, we have the isomorphism
$\Endo(V[1])\simeq \fr{s}\Endo(V)$. Now $\form$ gives a pairing between
$V[1]$ and $V[-1]$ so that  we get an isomorphism $\Endo(V[1])(n)\simeq
(\check V[-1])^{\otimes n}\otimes V[1]$. This space has natural degree $n-1$ and has a natural $\Snp$
action. Since all the degrees are shifted by one, we see that if $\form$ is
symmetric, $\fr{s}\Endo$ is anti--cyclic and if it is skew $\fr{s}\Endo$ is cyclic.
\end{ex}

\subsection{Na\"ive suspension and odd versions}
We again use a na\"ive shift, as in \ref{oddsec}, and define an odd cyclic operad to be the result of the na\"ive shift of an anti-cyclic operad.
In particular $\O$ is a cyclic operad if and only if $\Sigma \fr{s}\O$ is an odd cyclic operad.

\subsection{The tree picture, coinvariants and arbitrary finite sets}
Intuitively, operads correspond to rooted trees whereas cyclic operads
correspond to (non-rooted) trees.  This intuition can be made precise in the language of triples (see Section $\ref{triplepar}$).  There is an obvious forgetful functor from rooted trees to trees, which gives the inclusion of the operations
corresponding to a rooted tree into those of a cyclic operad.
The axioms of a cyclic operad guarantee that the
operation of a rooted tree is equivariant under changes of the root.

On the other hand given just a tree, to make it rooted,
requires a choice of a root.  If the tails are not labeled, there is no canonical choice.
The only thing to do is to sum over all of these choices which corresponds to using the operator $N$, i.e. passing to (co)invariants.

\begin{caveat}
\label{caveat}
Here there is one serious caveat. When composing along a rooted tree, with one edge say, one
can identify the set of flags $n\setminus \{i\} \amalg m$
with $n+m-1$ by first enumerating the first $n$ elements until $i$ is reached
then enumerating the $m$ elements of the second set and the rest of
the elements of the first set. That is the set above has a natural linear
order.

On the other hand, when composing along a non-rooted tree, as in the cyclic case, the set
$n\setminus \{i\} \amalg m\setminus \{j\}$ does not have a canonical linear
order, but only a cyclic one.
 If $j=0$ and $i\neq 0$, then we are in the case above
and we do have such an order. Likewise if $i=0$ and $j\neq 0$, we again
can make a linear order by switching the factors. This is essentially equivalent
to the condition in the definition of a cyclic operad.

Notice that things are completely unclear where both $i=0$ and $j=0$.  More on this below; see \S\ref{circijsec}.
\end{caveat}
These issues can be circumvented by working with coinvariants or with cyclic operads over arbitrary finite sets, as we presently recall.

\subsubsection{(Cyclic) Coinvariants}
Given a cyclic or anti-cyclic or odd cyclic operad $\O$ we define its space of coinvariants to be $\OSpp:=\bigoplus \O(n)_{\Snp}$.

We will also consider just the cyclic coinvariants $\CO_{C}^{\oplus}:=\bigoplus \O(n)_{C_{n+}}$ where $C_{n+}$ is the cyclic subgroup generated by $T$ in $\Snp$.  The cyclic coinvariants also make sense for a non-$\Sigma$ cyclic operad.

\subsubsection{Cyclic operads via arbitrary indexing sets}\label{afsets}
A nice way to think about cyclic operads is to look at operads in arbitrary sets. We think of the inputs and the output labeled by a set $S$.
 That is we get objects $\O(S)$ for any finite set $S$ together
with isomorphisms $\phi_*:\O(S)\to \O(S')$ for each bijection $\phi:S\to S'$.
As well as  structure maps
\begin{equation}
\scirct: \O(S)\otimes \O(T) \to \O((S\setminus\{s\})\amalg (T\setminus\{t\}))
\end{equation}
these maps are equivariant with respect to bijections and associative in the appropriate sense.

The cyclic or anti-cyclic condition then translates to
\begin{equation}
a\scirct b =\pm (-1)^{deg(a)deg(b)} b\leftsub{t}{\circ}_s a
\end{equation}
where the extra minus sign is present in the anti-cyclic case.

\subsubsection{Moving between the biased and un-biased pictures}

Given a cyclic operad $\O$, one sets
\begin{equation}
\O(S)=\left(\bigoplus_{\text{bijections } S \leftrightarrow \{0,1,\dots,|S|-1\}}\O(|S|-1)\right)_{\Snp}
\end{equation}
Where $\Snp$ acts diagonally on both the sum, by acting on the bijections, and the summands.
Given the full finite set version,
the  version using the natural numbers is basically given by inclusion.

For operads switching from $\O(n)$ to $\O(X)$
corresponds to switching from the category of finite sets with
bijections to its skeleton, the category with objects the natural numbers
and only automorphisms, where $n$ represents the set $\{1,\dots, n\}$
and $Aut(n)=\Sn$. For cyclic operads $n$ actually represents the set
 $\{0,1,\dots, n\}$ and $Aut(n)=\Snp$.

Following Markl, we will call the skeletal version involving
only the natural numbers the biased version. The finite set version
is then the un-biased one.

\subsubsection{Coinvariants and the un-biased setting}\label{coinvsec}

When working with finite sets, things become nicer on the level of coinvariants.  Here it suffices to take $\OSpp$.  The categorical proof is that this represents the colimit over the category of finite sets with bijections of $\O$ viewed as the functor that assigns $\O(S)$ to a set $S$.

A pedestrian way to say this is that taking coinvariants, we can first identify sets which are in bijection with each other and then only have to mod out by automorphisms.  For each finite set $S$ we can choose $\{0, \dots, |S|-1 \}$ as such a representative.

\subsection{The bracket in the anti-cyclic case}
\begin{df} Let $\O$ be an anti-cyclic operad
For $a\in \O(S)$ and $b\in \O(T)$ we define
\begin{equation}
[a\odot b]:=\sum_{s\in S,t\in T}a \scirct b
\end{equation}
\end{df}

\begin{prop}
\label{antibracketprop}
$\cLiebracket$ is anti-symmetric and satisfies the Jacobi
 identity for any three elements
in the sense that for $a\in \O(S),b\in \O(T),c\in \O(U)$
\begin{multline}
[a\odot b]=-(-1)^{deg(a)deg(b)}[b\odot a]  \in \bigoplus_{s\in S\,t\in T}
\O((S \setminus s)\amalg (T\setminus t))\\
\mbox{}(-1)^{deg(a)deg(c)}[a\odot[b\odot c]]+(-1)^{deg(a)deg(b)}[b\odot[c\odot a]]+
(-1)^{deg(c)deg(b)}[c\odot[a\odot b]]=0\\
\in \bigoplus_{s\in S,t\in T,u \in U}
\O((S \setminus s)\amalg (T\setminus t)\amalg (U\setminus u))
\end{multline}
\end{prop}

\begin{proof}
The proof is a straightforward calculation.
The first equation directly  follows from the
antisymmetry of the operations $\scirct$ for an anti--cyclic operad.

Checking the Jacobi identity is straight forward: $(-1)^{deg(a)deg(c)}[a\odot [b\odot c]]=$

$$\begin{array}{ll}
 &  (-1)^{deg(a)deg(c)} \ds\sum_{\substack{t^\prime\in T\coprod U \setminus\{t,u\} \\ s \in S}}
\ds \sum_{\substack{t\in T \\ u\in U}} a\xycirc{s}{t^\prime}(b\xycirc{t}{u}c) \\
=&  (-1)^{deg(a)deg(c)} \ds\sum_{\substack{t^\prime\in T \setminus\{t\} \\ s \in S}}
\ds \sum_{\substack{t\in T \\ u\in U}} a\xycirc{s}{t^\prime}(b\xycirc{t}{u}c)+ (-1)^{deg(a)deg(c)} \ds\sum_{\substack{t^\prime\in U \setminus\{u\} \\ s \in S}}\ds \sum_{\substack{t\in T \\ u\in U}} a\xycirc{s}{t^\prime}(b\xycirc{t}{u}c) \\
\end{array}
$$
$$
\begin{array}{ll}
=&  (-1)^{deg(a)deg(c)} \ds\sum_{\substack{t^\prime\in T \setminus\{t\} \\ s \in S}}
\ds \sum_{\substack{t\in T \\ u\in U}} (a\xycirc{s}{t^\prime}b)\xycirc{t}{u}c-(-1)^{deg(a)deg(b)} \ds\sum_{\substack{t^\prime\in U \setminus\{u\} \\ s \in S}}\ds \sum_{\substack{t\in T \\ u\in U}} (b\xycirc{t}{u}c)\xycirc{t^\prime}{s}a \\
=&  -(-1)^{deg(b)deg(c)} \ds\sum_{\substack{t^\prime\in T \setminus\{t\} \\ s \in S}}
\ds \sum_{\substack{t\in T \\ u\in U}} c\xycirc{u}{t}(a\xycirc{s}{t^\prime}b)-(-1)^{deg(a)deg(b)} \ds\sum_{\substack{t^\prime\in U \setminus\{u\} \\ s \in S}}\ds \sum_{\substack{t\in T \\ u\in U}} b\xycirc{t}{u}(c\xycirc{t^\prime}{s}a) \\
=& -(-1)^{deg(b)deg(c)}[c\odot [a\odot b]]-(-1)^{deg(a)deg(b)}[b\odot [c\odot a]]
\end{array}$$
\end{proof}
Notice that in this statement, we use the conventions stated in the beginning.

In view of \S\ref{coinvsec} the following theorem is now straightforward.
\begin{thm}
If $\O$ is an anti-cyclic operad then $\cLiebracket$ induces
a Lie bracket on $\OSpp$.
\end{thm}

We will denote this Lie bracket by the same symbol.

\begin{rmk}
Notice that unlike in the operad case, this bracket is {\em not the anti--symmetrization of a pre--Lie structure}. It is actually the choice
of the root that gives this extra structure in the operad case through the
linear orders on the compositions. Here
no such consistent choice for linear orders exists. See also \ref{caveat}
and \S \ref{circijsec}.
\end{rmk}

\begin{ex}\cite{Kthree,CV}
Fixing a sequence of vector spaces of dimension $2n$ with a symplectic
form on them, we immediately get three sequences of Lie algebras from
the anti--cyclic operads $Comm\otimes \Endo(V^n),Lie\otimes \Endo(V^n)$ and
$Assoc\otimes \Endo(V^n)$. These are exactly the three sequences
considered by Kontsevich in his seminal paper \cite{Kthree} and further studied
by \cite{CV}. There is also the generalization of this construction to cyclic quadratic Koszul operads \cite{Ginz}.
\end{ex}

\begin{ex}
Likewise we can fix a sequence of dimension $n$ vectors spaces $V^n$ with
a symmetric non--degenerate bilinear form and consider the sequence of Lie algebras obtained by  $pLie\otimes \Endo(V^n)$.
\end{ex}
This begs the
\begin{qu}
What is the underlying geometry in the $pLie$ case?
Or in the other cases of \cite{chapoton}?
\end{qu}

\begin{ex}
Of course any suspension of a cyclic operad will yield an anti-cyclic one and hence a Lie algebra
and any tensor product of a cyclic operad with an anti-cyclic one will give
and anti-cyclic operad and hence a Lie algebra.
\end{ex}

\subsection{Lift to the cyclic coinvariants, non-$\Sigma$ version}
As mentioned before, the set $n\setminus \{i\} \amalg m\setminus \{j\}$ has
no canonical linear order,  but it does have a cyclic order.
Hence we can identify it with $n+m-1$ up to the action of $C_{n+m-1+}$.
Using this identification, we can restrict to the $C_{n+}$ coinvariants of
the sets $n$ to obtain a bracket on the cyclic coinvariants and
since we are only taking $C_{n+}$ coinvariants it actually suffices to take a
non-$\Sigma$ cyclic operad.

\begin{thm}\label{r2}
If $\O$ is an anti-cyclic operad then $\cLiebracket$ induces
a Lie bracket on the cyclic co-invariants
$\CO_{C}^{\oplus}:=\bigoplus \O(n)_{C_{n+}}$.  The result also holds true for $\O$ a non-$\Sigma$ anti-cyclic operad.
\end{thm}

\begin{ex}
The necklace Lie algebra of Bocklandt and Le Bruyn \cite{BLB,Schedler}
is an example of such a Lie algebra structure. Here the cyclic operad
structure is on the oriented cycles and the necklace words are
the cyclic invariants.
\end{ex}

\subsection{The bracket in the biased setting and compatibilities}
\label{circijsec}
Using the above description, we can relate the original
brackets to those arising in the operad setting. The obstruction
is that the two brackets lift to different spaces, but we can
use the operator $N$ which maps $\O(n)$ to $\O(n)^{C_{n+}}$ to make
the connection.

We first introduce the operations

\begin{equation}
a\cyccirc{i}{j}b= T^{1-i}a\circ_1 T^{-j}b
\end{equation}

Notice that $a\circ_ib=a\cyccirc{i}{0}b$ and $b\circ_j a=a\cyccirc{j}{0}b$.  The cyclic bracket $\cLiebracket$ in the unbiased setting is then given by $\cLiebracket= \sum_{i,j}\cyccirc{i}{j}$.

\begin{prop}  Let $\O$ be an anticyclic operad.  The map $N$ induces a map of Lie algebras from $\O^{\oplus}_C$ with
bracket $\cLiebracket $ to $\O^\oplus$ with bracket $\Liebracket$ via $[a]\mapsto N(a)$.
\end{prop}

\begin{proof}  Let $a\in \O(n)$ and $b\in \O(m)$.  One may calculate from the axioms that
\begin{equation*}
\frac{[[N(a)]\odot [N(b)]]}{(n+1)(m+1)}=\frac{[N(a),N(b)]}{n+m}
\end{equation*}
and hence that,
\begin{equation*}
N([[a]\odot [b]])=N(\frac{[[N(a)]\odot [N(b)]]}{(n+1)(m+1)}=N(\frac{[N(a),N(b)]}{n+m}) = [N(a),N(b)]
\end{equation*}
In particular the operadic bracket is closed on the subspace of cyclic invariants.
\end{proof}
A similar relationship holds between the fully symmetric bracket and fully symmetric cyclic bracket.  One way to keep track of the relationships between all of these Lie brackets is to know which objects represent the Maurer-Cartan functors, and to understand the relationship between the representing objects.  E.g. the arrow $\O^{\oplus}_C\to\O^\oplus$ is simply a manisfestation of the fact that every Frobenius algebra is in particular an associative algebra; see \cite{Ward2}.

\subsection{The odd Lie bracket and odd cyclic operads}

We can now adapt Gerstenhaber's construction to the cyclic operad setting.

\begin{prop}
If $\O$ is a cyclic operad then $\fr{s}\O$ is an anti--cyclic operad
with a Lie bracket $\cLiebracket$. This Lie bracket yields an odd Lie bracket
$\cGbracket$
on $\OSpp$ when using the degree $|\,.\,|$. More precisely
it is an odd Lie bracket on  the odd cyclic operad $\Sigma \fr{s}\O$.
\end{prop}

\begin{proof}The only thing to check is  that the signs are correct. This follows
from the fact that the degree of $a$ in  $\fr{s}\O$ is indeed $s(a)=|a|-1$. In particular $\cLiebracket$ is
a Lie bracket for the grading $s$ and hence after applying a shift, again by Lemma \ref{shiftlem},
it is odd Lie for the grading $|\,.\,|$ which is given by an additional na\"ive shift.
\end{proof}

\section{Dioperads, (Wheeled) PROPs and Properads.}\label{disec}
In this section we consider several further generalizations of operad structures. For an operad $\O$ it is natural to consider $\O(n)$
 as having $n$ inputs and one output. The first generalization is to include multiple inputs and outputs. The next
 generalization is to allow non-connected graphs.  Allowing both of them one arrives at PROPs, which were actually first historically
 \cite{McL,BV}. Restricting back to the connected graphs, one arrives at the notion of properads \cite{Val}.

 The next step, which will take us to the realm of mantra 2 is to allow self--gluings. This leads to the notions of
 wheeled PROPs and wheeled properads \cite{MMS}. Here it will become
 apparent that the odd gluing is essential. For the wheeled cases there is still a shift, which will allow us to make the gluings odd.
 This is intimately related to the fact that PROPs just like operads have distinct inputs and outputs.

 Finally, wheeled PROPs as they deal with non-connected graphs are the first instance where a multiplication for the BV operator and the Gerstenhaber bracket naturally appears.

\subsection{PROPs}
A unital PROP in the biased definition has an underlying sequence of objects
$\O(n,m)$ of $\CalC$ or say $\dgVect$ which carry an $\Sn\times \SS_m$
action. For this collection of bimodules to be a PROP, it has to have
the following additional structures.

\begin{enumerate}
\item Vertical compositions $\boxbar:\O(n,m)\otimes \O(m,k)\to \O(n,k)$ which
are equivariant
\item Horizontal compositions $\boxminus:\O(n,m)\otimes \O(k,l)\to \O(n+k,m+l)$
which are compatible in the sense that $(a\boxbar b)\boxminus (c\boxbar d)
=(a\boxminus c)\boxbar (b\boxminus d)$
\item Unit. $\unit \in \O(1,1)$, s.t. $(\unit \boxminus \dots \boxminus \unit) \boxbar a=a\boxbar(\unit \boxminus \dots \boxminus \unit)=a $
\end{enumerate}

The collection of objects $\O(n,m)$ together with the $\Sn\times \SS_m$ action is called an $\SS$--bimodule.

We define the compositions $a \ccirc{i}{j}b$ by  adding identities in all slots other than into the input slot $i$ of $a$ and the output slot
$j$ of $b$ and gluing $i$ and $j$ together.
These operations, gluing one input to one output, are called dioperadic operations  $\ccirc{i}{j}:\O(n,m)\otimes \O(k,l)\to \O(n+k-1,m+l-1)$.

\begin{df}
A dioperad is a collection of $\SS_n\times \SS_m$ modules $\O(n,m)$ with the operations $\ccirc{i}{j}:\O(n,m)\otimes \O(k,l)\to \O(n+k-1,m+l-1)$
that are symmetric group invariant and associative.

An nc--dioperad is a dioperad together with a horizontal multiplication $\boxminus:\O(n,m)\otimes \O(k,l)\to \O(n+k,m+l)$ which is compatible with all
the other structures, like in a PROP.
\end{df}

\begin{figure}
\centerline{\includegraphics[width=2in]{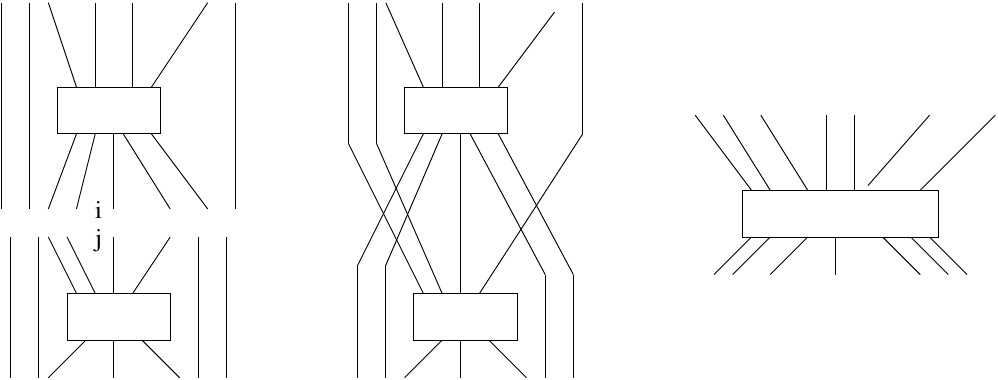}}
\caption{\label{diopfig}The dioperadic compositions}
\end{figure}

In the unbiased version one has a functor $\O$
from $\Fin\times \Fin$ to $\CalC$. Using the unit, one obtains
compositions $\ccirc{s}{t}:\O(U,S)\otimes \O(T,V)\to \O(U\amalg T\setminus \{t\}, V
\amalg (S\setminus \{s\})$.

\begin{ex}[Endomorphism PROP]
The canonical example is the endomorphism PROP $\Endo(V)(n,m)=Hom(V^{\otimes n},V^{\otimes m})$ with
the obvious $\Sn\times \SS_m$ action permuting the variables and functions together with the
obvious compositions.
\end{ex}

\begin{rmk}
\label{opproprmk}
Every PROP contains an operad given by the $\O(n,1)$ and the dioperadic operations $\ccirc{i}{1}=:\circ_i$.
\end{rmk}

\begin{ex}
\label{propgenopex}
[PROP generated by an operad]
An operad can be thought of as giving a sequence $\O(n,1)$.
Setting

\begin{equation}
\label{oppropeq}\O(n,m):=\bigoplus_{(n_1,\dots, n_m): \sum n_i=n} \O(n_1)\odo \O(n_m)\times_{\SS_{n_1}\times \dots\times \SS_{n_m}\times \SS_m}\Sn\times \SS_m
\end{equation}
The $\SS_m$ action permutes the factors and
the $\Sn$ action acts via the identification of the disjoint union of the sets $\{1,\dots,n_i\}$  with
the set $\{1,\dots,n\}$ by first
enumerating them one after another in the order given by $i$ that is via the representation induced by the inclusion  $\SS_{n_1}\times\dots\times \SS_{n_m}\to\SS_n$.
We obtain a PROP by defining $\boxminus$ to be essentially the identity, i.e.\ just tensoring together
the two factors followed by the inclusion of the summand.
This is a good example of a non--connected generalization
treated in \S\ref{NCpar}.
\end{ex}

\subsubsection{Properads}
Looking at the definition of a PROP one can see that the associativity implies
that there are compositions defined for any  oriented graph $\Gamma$, see \cite{Val,MV} for details.

Restricting to the situation where compositions are defined for
all {\em connected} oriented graphs
 one obtains the notion of a properad \cite{Val}. For instance
  the horizontal composition $\boxminus$ is dropped.

\subsubsection{Algebras}
An algebra over a PROP(erad) $\O$ is then a vector space $V$ together with a morphism of PROP(erad)s $\O\to \Endo(V)$

\subsubsection{Coinvariants} We let $\Op_\mathbb{S}$ be the sum over the coinvariants $\O(n,m)_{\Sn\times \SS_m}$.
\subsection{Poisson--Lie bracket}
Analogously to the structure of the Lie bracket for operads,
we can define for $a\in \O(n,m)$ and $b\in \O(k,l)$

\begin{equation}
a\circ b:=\sum_{i,j} a\ccirc{i}{j}b, \quad  [a\circ b]:=a\circ b-(-1)^{deg(a)deg(b)}b\circ a
\end{equation}

As before we let $\Op=\bigoplus_{n,m}\O(n,m)$ and $\Op_{\SS}:=\bigoplus_{n,m}\O(n,m)_{\Sn\times \SS_m}$.
In the case of a PROP, we also have a natural multiplicative structure
given by $\boxminus$.
\begin{thm} For a PROP(erad) or a dioperad $\O$ ,
the product above is Lie admissible on $\Op$, that is, the graded commutator induces a Lie bracket $\Liebracket$.
This Lie bracket descends to $\Op_{\SS}$.

For a PROP $\O$ or an nc--dioperad, the induced Lie bracket on  $\Op_{\SS}$ is Poisson w.r.t $\boxminus$.

The Lie bracket for an operad induces a Poisson bracket
on the  PROP generated by that operad coinciding with the natural Poisson
bracket above.
\end{thm}

\begin{proof}
The proof of the Lie--admissible structure and hence Jacobi identity can be adapted from proof of Proposition \ref{antibracketprop} for the anti--cyclic operad case.
For this we have to partition sets $S,T,U$ into in and outputs and restrict the sum only over the in--to--out gluings. The sign for switching the order is the same as using the anti--commutator and the same six terms that cancel appear.

To show the Poisson
property, we note that
\begin{equation}
\label{nceq}
a \circ (b \boxminus c)=(a\circ b)\boxminus c + (-1)^{deg(a)deg(b)} b\boxminus(a\circ c)
\end{equation}
up to symmetric group actions depending if an output of $a$ is glued to $b$ or $c$, and where the sign comes from the commutativity constraint in $\gVect$. The last statement follows by the definition of
the Poisson property and Example \ref{opproprmk}.
\end{proof}

 Adding a vertical composition formally to properads, by using not necessarily connected graphs, we end up back with PROPs.
For cyclic operads things are a bit more complicated, and we
have to first introduce the notion of non--connected cyclic operads.  This is done in \S\ref{NCpar}.

\subsection{Odd versions}
The odd versions of the concepts above can again be defined by using
shifts and suspensions.

\subsubsection{Suspension}
The {\em suspension} of an (nc)-dioperad or PROP(erad) $\O$ is defined to be $\fr{s}\O:=\O\tensor\Endo(\Sigma^{-1}k)$.  Explicitly the underlying $\SS$-bimodule
is
\begin{equation}
\fr{s}\O(n,m)=\Sigma^{n-m}\O(n,m)\otimes (sgn_n\otimes sgn_m)
\end{equation}

Just like for operads we have the following version of Proposition \ref{shiftprop}:

\begin{prop}\cite{MSS}
$V$ is an $\O$--algebra if and only if $V[1]$ is an $s\O$ algebra.
\end{prop}

\subsubsection{Na\"ive/output shift}
The notion of a `na\"ive' shift is no longer so na\"ive.  We can again take $\Endo$ as a guide.
Naively shifting it as an operad and then taking the PROP it generates we are led to the following definition.

Given an $\SS$-bimodule $\O$, we let $\sout\O$ be the bimodule
\begin{equation}
\sout\O(n,m)=\Sigma^{m}\O(n,m)\otimes sgn_m
\end{equation}
Just like in the case of operads (which is a subcase), one obtains slightly different signs in the associativity equations than one would
expect for the induced operations.

\begin{df}
An odd (nc)--dioperad or PROP(erad) is the na\"ive shift of the structure. That is $\O$ is an odd PROP(erad) if and only if $\sout^{-1} \O$ is a PROP.
\end{df}

\begin{ex}
An example of such an odd PROP(erad) is given by
\begin{equation}\O(n,m)=
\check V^{\otimes n}\otimes \Sigma^m   (V^{\otimes m} \otimes sgn_m)
\end{equation}
with the natural $\Sn\times \SS_m$ action.
The
vertical composition  given by the natural
pairing are given by the natural
pairing $\check V\otimes V\to k$
and the horizontal composition is induced by tensoring together the factors.
\end{ex}

We will also consider the suspension
given by $\sin\O(n,m)=\Sigma^n sgn_n \otimes \O(n,m)$. With this notation,
we see that $\fr{s}=\sin \sout^{-1}$.

With these notions in place, we can use the mantra (1) and using (3)
due to the existence of $\boxminus$ the resulting bracket is more over Gerstenhaber.

\begin{thm}\label{r3}
An odd (nc)-dioperad or PROP(erad)  $\O$ carries an odd Lie bracket on $\Op$ and $\Op_{\SS}$.
The odd Lie bracket is Gerstenhaber w.r.t.~$\boxminus$ for an odd nc--dioperad or PROP on $\Op_{\SS}$.
The odd Lie bracket on an odd operad induces an odd Lie bracket
on the odd PROP generated by that operad and it is a Gerstenhaber bracket there.
\end{thm}

\begin{proof}
The only thing to check is that the effective shift for the dioperadic operations is indeed one.
This is the case, since before the dioperadic operation, the total shift is $n+m$ and after the
shift it is $n+m-1$.
\end{proof}

\subsection{Wheeled versions}
The dioperadic operations  and $\boxminus$ are not quite enough to recover the PROP structure. After one such operation, to get
to the operation $\boxbar$ one would have to do self--gluings of one input to an output.
This is precisely what is allowed in the wheeled version.

That is, in the unbiased version a wheeled PROP has the operations $\boxminus$, $\scirct$
 and self-gluing operations $\circ_{st}:\O(S,T)\to \O(S\setminus \{s\},T\setminus \{t\})$ which again satisfy
  natural equivariance, associativity and compatibility.  The compositions are defined for not necessarily connected oriented graphs with wheels.

Dropping the horizontal composition $\boxminus$ one obtains the notion of a wheeled properad.
The compositions are defined for  {\em connected} oriented graphs with wheels.
Notice that since $\boxbar$ can now be reduced to single self--gluings and dioperadic gluings,
the notion of wheeled dioperad would coincide with wheeled properad and that of wheeled nc--dioperad with that
of a wheeled PROP.

\begin{ex}
The PROP(erad) $\Endo_V(n,m)\simeq \check V^{\otimes n} \otimes V^{\otimes m}$ has such a
natural wheeling by simply contracting tensors for the self--gluings.
\end{ex}

\subsubsection{Wheeled odd PROP(erad)s}
\label{BVproppar}
The odd versions are described just as above. These are
by definition the images under the suspension $\sout$.
Again, we denote the image of the compositions $\ccirc{i}{j}$ and $\circ_{ij}$ by $\cbullet{i}{j}$ and $\bullet_{ij}$.

\begin{lem}
\label{oddselflem}
In an odd wheeled PROP(erad), we have $\bullet_{ij} \bullet_{kl}(a)=-\bullet_{k'l'} \bullet_{i'j'}(a)$, where $i',j',k',l'$ are the names of the
appropriately renumbered flags.
\end{lem}

\begin{proof}
This is due to the shift. Now if we interchange the order, we interchange outputs $j$ and $l$ resulting in a minus sign.
Since the inputs are unaltered, switching $i$ and $k$ gives no sign.
\end{proof}

This is the first time we encounter odd--self gluings, and we indeed
find the first occurrence of mantra (3).

\begin{thm}\label{r4}
For an odd wheeled PROP(erad) $\O$,
the operator $\Delta$ defined on each $\O(n,m)$  by
\begin{equation}
\Delta(a):=\sum_{ij} \bullet_{ij}(a)
\end{equation}
satisfies $\Delta^2=0$.

Moreover on the coinvariants for a PROP the operator $\Delta$ is a
BV operator on $\Op_{\SS}$ for the multiplication $\boxminus$ and
its associated bracket  (see Appendix)
is the Gerstenhaber bracket induced by $\Gbracket$.

\end{thm}

\begin{proof}
The reason for the vanishing of $\Delta^2$ is Lemma \ref{oddselflem}. For the BV bracket we notice that $\Delta(a\boxminus b)$
splits into four sums depending on the gluing. The inputs of $a$ glued
to the outputs of $a$ gives $\Delta(a)b$, the inputs of $b$ to the outputs of $b$ gives the term $a\Delta(b)$, the outputs of $a$ to the inputs of $b$ and vice--versa gives $a\bullet b$ and $b\bullet a$ respectively --- all up to permutations.

 The only thing that remains to be checked
is that the signs work out which they do by a straightforward computation.
$\Delta$ has degree $1$ since
 each $\bullet_{ij}$ has degree $1$ after the shift.
Finally, the structures descend as we sum over all possible gluings.
\end{proof}

\begin{rmk}
Notice that there is no BV in the unshifted case. We need the odd composition
to get a differential.  This also shows that the Gerstenhaber bracket is actually the deeper one and the regular Lie bracket is actually a shift of the odd one rather than vice-versa.
\end{rmk}

\section{Modular operads, triples and twisting}
\label{triplepar}

We will now turn to the notion of modular operads. This is the first case where the odd
version is not given by a simple shift or suspension. It is rather a twist, namely what is know as  a $\K$--modular operad. For this we will need
to introduce triples. With hindsight, we will see that all the operad-like structures and their corresponding odd versions discussed above also arise from triples and twisted triples.

\subsection{Modular operads}
We will introduce modular operads in the unbiased setting.

A modular operad is a collection $\O(g,S)$
 bi--indexed by finite sets and the natural numbers, usually
taken with the condition that $2g+2-|S|>0$ together with gluing maps
\begin{equation}
\scirct: \O(g,S)\otimes \O(g',T)\to
\O(g+g',S\setminus \{s\}\amalg T\setminus \{t\})
\quad \forall s\in S, t\in T
\end{equation}
and self gluing maps
\begin{equation}
\circ_{ss'}:\O(g,S)\to \O(g+1,S\setminus \{s,s'\})\quad  \text{ for all distinct } s,s'\in S
\end{equation}
which are compatible associative and equivariant with respect to bijections.
The details of these conditions are straightforward, but tedious and
we refer to \cite{GK,MSS}. An alternative definition utilizing
triples is below.

\begin{ex}
The motivating example are the Deligne--Mumford compactifications $\bar M_{g,S}$
of curves of genus $g$ with $|S|$ punctures labeled by the set $S$.  A linear example is then given by the $H_*(\bar M_{g,n})$.
\end{ex}
For the biased version, just like in the cyclic case, one uses the sets  $\{0,1,\dots, |S|-1\}$ and  the notation
$\O((g,n)):=\O(g,n-1)$.

\subsection{The free-forget adjunction}
Before delving into the categorical depth of triples, we will consider a relevant example in the case of operads coming from the free-forget adjunction. Given an operad $\O$ we can forget the gluing maps and only retain the $\SS$-module. This gives a functor $G$ between the respective categories. The functor
$G$ has a left adjoint $F$, the free operad functor.
Explicitly, given an $\SS$--module $\V$, the free operad $F(\V)$ on $\V$
is constructed as follows. For a rooted tree $\tau$ one sets
\begin{equation}
\label{vtau}
\V(\tau)=\bigotimes_{v \text{ vertex of $\tau$}} \V(In(v))
\end{equation}
where $In(v)$ is the set of flags or half edges incoming at $v$.
Recall that in a rooted tree there is a natural orientation towards
the root and this defines the outgoing edge or flag at each vertex. All
other flags are incoming.

The composition $\circ_{\tau}$ is obtained by contracting all edges, that is for each edge we perform a $\circ_i$ operation
where $i$ is the input flag of the edge.

Rooted trees whose tails are labeled by a set $S$  form a category $\Iso\RT(S)$, by allowing isomorphisms of labeled rooted trees
as the only morphisms.
The free operad is then given by the $\SS$--module

\begin{equation}
F(\V)=\colim_{\Iso\RT(n)}\V=\bigoplus_{\tau\in \RT(n)}\V(\tau)/\sim \;\;=\bigoplus_{[\tau] \text{ iso classes}}\V(\tau)
\end{equation}
where $\sim$ is the equivalence under push--forward with respect to isomorphism.
The operad structure on the $F(\V)(S)$ is given summand by summand. If there are two summands indexed by $\t$ in $F(\V)(S)$  and $\t'$ in $F(\V)(T)$
under the composition $\circ_s$ their tensor product maps to the summand $\t\circ_{s}\t'$ which is the tree where $\t'$ is glued onto $\t$ at the leaf indexed by $s$.

\begin{table}
\begin{tabular}{l|l|l}
type&graphs for triple&local data at a vertex $v$\\
\hline
operad&rooted trees&in flags\\
non--$\Sigma$ operad&planar rooted trees&in flags\\
cyclic operad&trees&flags\\
non-$\Sigma$ cyclic operad&planar trees&flags\\
modular operad&stable graphs&(flags, $g(v)$)\\
PROP& nc directed graphs without wheels &(in flags, out flags)\\
properad&connected directed graphs without wheels&(in flags, out flags) \\
dioperad&directed trees&(in flags,out flags) \\
nc-dioperad&directed forests&(in flags, out flags) \\
wheeled PROP& nc directed graphs with wheels&(in flags, out flags)\\
wheeled properad&connected directed graphs with wheels&(in flags, out flags)\\
\end{tabular}

\caption{\label{optable}Types of operads and the graphs underlying their triples. nc stands for not necessarily connected}
\end{table}

\subsection{Operads and triples}
Let $\T=GF$ which is an endo--functor from $\SS$--modules to $\SS$--modules.
Since $F$ and $G$ are an adjoint pair, there are natural transformations: $\epsilon:FG\to id$ and $\eta: id \to GF$.
Vice--versa, one can prove that $F$ and $G$
  indeed form an adjoint pair using these natural transformations; see e.g.\ \cite{GelMan}.
In our particular case, the first is given by sending the summand of $\t$ to its image under the composition $\circ_{\tau}$. This
is well defined up to isomorphism because of the equivariance of the gluings.
 The second is just inclusion of the summand given by the $S$ labeled tree with one vertex.

\subsubsection{Triples}
Using these on $\T$ one gets the following natural transformations $\mu:\T\T\to\T$ via $G(FG)F\stackrel{\epsilon}{\rightarrow}GF$ and
$\eta: id\to \T$, making $\T$ a (unital associative) monoid.  In general {\em a triple} is an endo-functor $\T$ together with $\mu$ and $\eta$ which satisfies just these equations.  Our triple was constructed using an adjoint pair and it is a fact that all triples actually arise this way \cite{triple1,triple2}.

\subsubsection{Operads}
Now if $\O$ is an operad, we also get a map $\alpha:\T\O\to \O$ by sending each summand $\O(\tau)$ indexed by an $S$--labeled tree $\t$ to $\O(S)$ using $\circ_{\tau}$.
Due to the associativity these maps satisfy the module equations when considering the two possible ways to map $\T\T\O$ to $\O$.

Vice--versa, given an $\SS$--module $\V$ if we are given a morphism $\alpha:\T\V\to \V$, we have equivariant maps $\circ_{\t}$ and
moreover if they satisfy the module equations, then these $\circ_{\tau}$ decompose into
elementary maps $\circ_s$, where the
$\circ_s$ come from rooted  trees with exactly one internal edge.  It is straightforward to check that the $\circ_s$ define an operad structure on the $\V(S)$.

The natural transformation $\mu$ also has a nice tree interpretation. Let $\tau_0$ be the tree index of the first application of $\T$,
then in the next application one picks up a collection of indices $\tau_v$, one for each vertex $v$ of $\tau_0$. In order to show
the associativity, one can see that the corresponding summand of $\T\T\V$ is the same as $\V(\t_1)$ where $\t_1$ is obtained
from $\t_0$ by blowing up each vertex $v$ into the tree $\t_v$. Vice--versa, $\tau_0$ is obtained from $\tau_1$ by contracting the subtrees
$\tau_v$ to a vertex. One sometimes writes $\tau_1\to \tau_0$ since this is a morphism in the na\"ive category of graphs.

\subsubsection{Algebras over triples}
In general an algebra over a triple $\T$ is an object $\V$ of the underlying category together with a map $\alpha:\T\V\to \V$ such
that $\alpha,\mu$ and $\eta$ satisfy the axioms of a module over an algebra with a unit; see
\cite{MSS} for the precise technical details. From the above, we obtain:
\begin{prop}
Operads are precisely algebras over the triple $\T$ of rooted trees. \qed
\end{prop}

\subsection{Other cases}

The method is now set to define all the other cases as algebras over a triple. We only have to specify the triple.
Taking the cue from above, we have to (1) fix the type of graph and the category of isomorphisms,
 (2) fix the value of $\V$ on each graph, i.e. the analogue of equation (\ref{vtau}); in all common examples this is local
 in the vertices, then (3)
 set $F(\V)=colim_{\Iso\G} \V$ where the colimit is taken over the category of isomorphisms of $S$--labeled graphs of the given type and
 (4) give $\mu$ via gluing
 the graphs together by inserting the graphs indexed by a vertex into that vertex.
Think of this as the blow--up which is  inverse to the operation of
 contracting the subgraph.

 For (1) we use Table \ref{optable} where we take the $S$--labeled version of the respective graphs.
 For (2) we use the general formula
 \begin{equation}
 \label{ogammaeq}
 \O(\Gamma)=\bigotimes_{v \text{ vertex of $\Gamma$}} \O(loc(v))
 \end{equation}
 where $loc(v)$ is the local set at $v$ given in Table \ref{optable},
 and for (4) we use the gluing together of flags; see appendix.

 Notice that in each of the examples the underlying objects are graphs of some sort. These form a na\"ive category of
 graphs, by allowing isomorphisms and contractions of edges, with the respective change of data.
 For modular operads for instance, when contracting a loop edge, one also has to increase the genus by one.

 \begin{prop}\cite{GeK1,GK,MSS,Val,Markl,MMS}
 The types of operads listed in Table \ref{optable} are precisely algebras over the respective triple defined above.
 \end{prop}

We will make this explicit for modular operads. Here the graphs are stable  $S$--labeled graphs, which means that they are arbitrary graphs together with a labeling by $S$ of the tails and a genus function $g$ from the vertices of the given graph $\Gamma$ to $\N$,
such that $2g(\Gamma)-2-|S|>2$ where $g(\Gamma)=\sum_{\text{vertices } v} g(v)+dim H^1(\Gamma)$ is the total genus of the graph.
The basic gluings  $\scirct$ come from trees with one edge where $s$ and $t$ are the flags of the unique edge and the gluings
$\circ_{ss'}$ come from the one vertex graph with one loop whose flags are indexed by $s$ and $s'$.

For various gradings the following formula is useful for an $S$--labeled $\Gamma$
\begin{equation}
\sum_{v}(|Flags(v)|-2+g(v))=2g(\Gamma)-2+|S|
\end{equation}

\subsection{Twisted modular operads}
The idea will be to get new notions of operad-like structures by twisting the triple $\T$.  Let us first recall how this is done when $\T$ is the triple encoding modular operads, following\cite{GK}.  In order to do this one first defines
\begin{equation*}
\V_{\D}(\Gamma):=\V(\Gamma)\otimes \D(\Gamma)
\end{equation*}
for a suitible datum $\D$, and then,
\begin{eqnarray}
\T_{\D}\V(g,S)&:=&\colim_{\Gamma\in \Iso\G_{mod}(S)}\V_\D
\nn \simeq  \bigoplus_{\Gamma\in \Iso\G_{mod}(S)}\V(\Gamma)\otimes\D(\Gamma)/\sim\nn   \\
&\simeq& \bigoplus_{[\Gamma]}(\V(\Gamma)\otimes \D(\Gamma))_{Aut(\Gamma)}
\end{eqnarray}
where here $\G_{mod}(S)$ are $S$--labeled stable graphs with a genus function and the last sum is over isomorphism classes of such graphs.
Taking coinvariants with respect to the automorphism group is new, since the automorphism groups of rooted $S$--labeled trees
are trivial.

In order for this to work $\D$ has to be what is called a hyper--operad in \cite{GK}. The relevant problem being that if we do the inverse of
contracting edges along subgraphs ---so as to build the composition along a graph--- we have to know how $\D$ behaves. So let $\Gamma_1$ be a
stable  graph and $\Gamma_0$ a graph obtained from $\Gamma$ by contracting subtrees $\Gamma_v$, where $v$ runs
through the vertices of $\Gamma_0$ and $\Gamma_v$ is the preimage of $v$ under the contraction. This is also what is needed
to define the transformation $\T_{\D}\T_{\D}\to \T_{\D}$.

The datum of $\D$ is given by specifying all the $\D(\Gamma)$ and maps
\begin{equation}
\D(\Gamma_0)\otimes \bigotimes_{v {\text{ vertices of $\G$}}}\D(\Gamma_v)\to \D(\Gamma_1)
\end{equation}
for each morphism $\Gamma_1\to \Gamma_0$
which again have to satisfy some natural associativity, see \cite{MSS,GK}.  One also fixes that $\D(*_{g,S})=k$, where $*_{g,S}$
is the graph with one vertex of genus $g$ and $S$ tails. These are necessary to show that the twisted objects are again triples with unit.
Notice that there might be no contractions of edges in $\Gamma_1\to \Gamma_0$. For this  subcase we have that $\D$ is compatible
with the $\Sn$ action.

\subsubsection{Compositions in twisted modular operads}
A good way to understand twisted modular operads is as follows. For a modular operad the algebra over a triple picture says
that for each $S$--labeled graph  $\Gamma$ with total genus $g$ there is a unique operation $\circ_{\Gamma}$ from $\O(\Gamma)\to \O(g,S)$. Now for a twisted modular operad this ceases to be the case. One actually has to specify more information on the graph.
One way to phrase this is that $\D(\Gamma)$ is a vector space of operations for each graph $\Gamma$ and we get a well defined
operation when we specify an element of that vector space. Of course basis elements suffice.
To make this precise, we use adjointness of $\tensor$ or in other words the fact that the category is closed monoidal.

\begin{lem}
\label{twistedcomplem}
Being an algebra over a $\D$ twisted triple in a closed monoidal category
 is equivalent to having equivariant, compatible composition maps
\begin{equation}
\circ_{ord,\Gamma}:\O(\Gamma)\to\O(g,S)
\end{equation}
for each $S$--labeled $\Gamma$ of total genus $g$ and each element $ord\in \D(\Gamma)$.
\end{lem}

\begin{proof}
The triple gives compatible compositions maps
$
\phi:\D(\Gamma)\otimes \O(\Gamma)\to \O(g,S)
$
that is $$\phi\in Hom(\D(\Gamma)\otimes \O(\Gamma),\O(g,S))\simeq Hom(\D(\Gamma), Hom(\O(\Gamma),\O(g,S) ))$$
In other words if $ord\in \D(\Gamma)$ then we get a composition $\circ_{ord,\Gamma}:\O(\Gamma)\to \O(g,S)$ and the collection
of these compositions is equivalent to $\phi$.
\end{proof}

\subsubsection{Coboundaries} A special type of twist is given by a functor from the one vertex graphs to
invertible elements in the target category. In the main application, this means a one--dimensional vector space
in some degree. That is a collection of $\mathfrak{l}(*_v)$ for each possible vertex type functorial under automorphisms;
in the modular case the vertex types are given by $(g,S)$ and the automorphisms are $\SS_{|S|}$.

If $\Gamma$ has total genus $g$ and tails $S$, then
$$\mathfrak{D}_{\mathfrak{l}}(\Gamma)=\mathfrak{l}(g,S)\otimes
\bigotimes_{v\in \Gamma} \mathfrak{l}((g(v),Flag(v))^{-1}$$

The most common coboundaries are listed in Table \ref{cobtable}.
\begin{table}
\begin{tabular}{lll}
Name&Value on $*_{((g,n))}$&appears in\\
\hline
$\mathfrak{s}$&$\Sigma^{-2(g-1)-n}sgn_n$&operadic suspension\\
$\tilde{\mathfrak s}$&$\Sigma^{-n}sgn_n$&shifts of $\cEndo$\\
$\Sigma$&$\Sigma k$&na\"ive shift\\
\end{tabular}
\caption{\label{cobtable}List of coboundary twists and their natural habitats  Here $n$ refers to  the standard notation
${\mathcal O}((n))={\mathcal O}(n-1)$ with $\SS_{(n-1)+}\simeq \SS_n$ action in the cyclic/modular case. }
\end{table}

Coboundaries behave nicely with respect to conjugation:
if $\mathfrak{l}$ is the functor of tensoring with $\mathfrak{l}$
then
\begin{equation}
\mathfrak{l}\circ {\mathbb T}_{\mathfrak D}\circ \mathfrak{l}^{-1}
\simeq \mathbb T_{\mathfrak D{\mathfrak D}_{\mathfrak l}}
\end{equation}
where we shall write $T_{\mathfrak D{\mathfrak D}_{\mathfrak l}}$ for $T_{\mathfrak D \tensor {\mathfrak D}_{\mathfrak l}}$.  This equation also proves
\begin{prop}
\label{twistprop}
The categories of algebras over the triple $\T_{\D}$ and algebras over the triple $\T_{\D\D_{\mathfrak l}}$
are equivalent, with the equivalence given by tensoring with $\mathfrak{l}$.
\end{prop}
This is the underlying reason for the form of our definition of odd operads and PROP(erad)s; see \S\ref{oddkpar}.

\begin{rmk}
\label{twistrmk}
It is important to notice that although $\mathfrak l$ determines $\D_{\mathfrak l}$, it can happen that different $\mathfrak{l}$ give
rise to the same twist $\D$. For instance $\D_{\s^2}\simeq \unit$\cite{GK}.
\end{rmk}

\subsection{The twist $\fr{K}$}  Given a (graded) finite dimensional vector space $V$ or an edge $e$ composed of two flags $s$ and $t$ let ${\text Det}(V)=\Sigma^{-dim(V)}\Lambda^{dim(V)}V$.  The most important feature about $Det$ is:
\begin{equation*}
\label{deteq}
Det(\bigoplus V_i)=\bigotimes_i Det(V_i)
\end{equation*}
We then define the hyper-operad (aka twist)
\begin{equation*}
\mathfrak{K}(\Gamma)=\text{Det}(\text{Edge}(\Gamma))
\end{equation*}

Algebras over the associated twisted triple will be called $\K$-modular operads, and are the correct odd version of modular operads. They
 turn up naturally in two situations. The first is as the Feynman transform of a modular
operad, (see \S \ref{feynmanpar}) and the second is on the chain and homology level of modular operads
with twist gluing or a degree one gluing; see \S\ref{toppar}.

\begin{lem}\label{ksqrmk}
$\K^{\otimes 2}=\D_{\mathfrak s}\D^{-1}_{\tilde{\mathfrak s}}$ and hence twists by $\K$ and $\K^{-1}$ are equivalent.  Explicitly, $\K^{\otimes 2}(\Gamma)=\Sigma^{-2|E(\Gamma)|}$.  In particular if we are only looking at the $\Z/2\Z$ degree then $\K=\K^{-1}$
\end{lem}

\subsubsection{Odd edge interpretation of $\K$}
The interpretation  which explains why $\K$-modular operads are the odd version of modular operads
is that in a $\K$-modular operad each edge gets weight $-1$ and so permutations
of the edges give rise to signs. Also permuting the vertices of an edge, gives the shifted sign.
These are exactly the Gerstenhaber signs as we discuss below \S\ref{oddkpar}.

\begin{rmk}  In the context of modular operads there are several additional interesting twists which we will not discuss at length here.  For example in considering the extension of the operad $\Endo(V)$ to the modular case.
one requires $V$ to have a non-degenerate form.  If the form is of degree $l$ and symmetric or anti-symmetric
the resulting operad structure is a twisted modular operad where the twists are; (see e.g.~\cite{Ba}):

\begin{eqnarray*}
{\mathfrak K}^{\otimes l}&\text{if the form is symmetric of degree $l$}\nn\\
{\mathfrak K}^{\otimes l}\D_\fr{s}&\text{if the form is anti-symmetric of degree $l$}
\end{eqnarray*}
These operads are then the natural receptacle in the formulation of an algebra over an operad.
\end{rmk}

\subsubsection{Tensor products}
\begin{lem}\label{tensorlem}
If $\O$ is a $\D$ twisted modular operad and $\O'$ is a $\D'$  twisted
modular operad then $(\O\otimes \O')((g,n)):=\O((g,n))\otimes \O'((g,n))$
is a $\D\D'$ twisted modular operad. \qed
\end{lem}

\subsection{Generalization of twists}
The theory of twisted triples works equally well for the other triples in Table \ref{optable}. In all these cases one has to specify
the following things. First, what the category of graphs is. This is given by contractions of edges and in the non--connected case also
by so called mergers, where two vertices are fused together keeping all inputs and outputs; see Appendix. Furthermore one has to specify a vertex type $*_{\Gamma}$ for each graph,
such that the component $[\Gamma]$ of the morphism
$\T\O\to \O$ yields $\circ_{\Gamma}:\O(\Gamma)\to \O(*_{\Gamma})$. Equivalently the
morphism $\T\T\to \T$ expands a vertex $*_{\Gamma}$ to all graphs with that vertex type.
In all
the cases there is a canonical choice given by the result of a total contraction of all edges followed by a total merger \cite{BM}.

Again as in Lemma \ref{tensorlem}, tensoring together twisted versions tensors the twists.

\subsubsection{Odd and anti-- as coboundaries}
Notice that the twists $\K$ always make sense and $\mathfrak s$ for the cyclic situation.
If we restrict $\K$ to trees, we find that the twist by $\K$ is precisely the twist by $\D_{\Sigma}\D_\fr{s}$. But the shift
$\Sigma \fr{s}$ was exactly what we associated to the grading of the Hochschild complex. Hence with hindsight, we could have
worked with $\K$ twisted operads and $\K$ twisted cyclic operads.

More precisely we have the  list of operad-like types given in  Table \ref{oddtable},
which could equivalently be defined as algebras over twisted triples.

\begin{table}
\begin{tabular}{lllll}
Type&defining &value (of $\mathfrak{l}$ & on& isomorphic twist\\
&twist&if coboundary)&&\\
\hline
odd operad&$\D_{\Sigma}$&$\Sigma k$&$*_n$&$\D_{\Sigma s}\simeq \K$\\
odd cyclic operads&$\D_{\Sigma \fr{s}}$&$\Sigma^{n-1} sgn_n$&$*_{((n))}$&$\K$\\
odd (wheeled) PROP(erad)&$\D_{s_{out}}$&$\Sigma^m sgn_m$&$*_{n,m}$&$\K$\\

$\K$--modular&$\K$&$Det(Edge)$&$\Gamma$&$\K$\\
\end{tabular}
\caption{\label{oddtable} Types of odd structures defined by  twisted triples via Proposition \ref{twistprop},  for twists that are isomorphic to $\K$. The twists
for the nc versions are the same.}
\end{table}

For (cyclic) operads, we have already clandestinely encountered these twists. Namely,
the odd (cyclic) operads are nothing but algebras of the triple of rooted trees, (respectively trees),  twisted by $\D_{\Sigma \fr{s}}$.
One can also check that anti-cyclic operads are equivalent to algebras over the triple twisted by $\D_\fr{s}$.
See Lemma \ref{twistisolem} for the proofs.

\begin{lem}
\label{twistisolem}
We have the following isomorphisms:
For operads $\D_{s}\simeq \unit$ and all the isomorphisms listed in Tables \ref{oddtable}.
\end{lem}

\begin{proof}

$\mathfrak{D}_s$ is concentrated in degree $0$ and the $\Sn$ action
is trivial. Indeed for an $n$--tree the shift is $n-1+\sum_v (1-ar(v))=n-1+|V|-|E_{int}|+n=0$.

For $\D_{\Sigma \fr{s}}$ the value on an $S$ labeled rooted tree
is $\D_{\Sigma \fr{s}}(T)=Det^{-1}(S)\otimes\bigotimes_v Det(In(v))\simeq Det(Edge)=\K(T)$.  Similarly for non-rooted trees.
Finally, for the PROP(erad)s for $\Gamma$ of type $(n,m)$ that is $n$ inputs and $m$ outputs
$\D_{s_{out}}(\Gamma)=Det^{-1}(Tail_{out}(\Gamma))\otimes \bigoplus_v Det(Flag_{out}(v))\simeq Det(Edge)\simeq \K$
where we used that the set of non-tail flags is in bijection with the edges.
\end{proof}

\begin{lem}
Notice that for PROP(erads) by an analogous argument $\D_{s_{out}}\simeq \D_{s_{in}}\simeq \K$ so that
$\D_{\fr{s}}\simeq \D_{s_{in}}\D_{s_{out}}^{-1}\simeq \unit$. Thus a suspended PROP(erad) is a PROP(erad). \qed
\end{lem}
\begin{rmk}
In \cite{MMS} the following cocycles are also used: $\mathfrak{s}=s^{-1}$, $w=\K^{-1} \mathfrak{s}$.
It seems although stated differently, that in \cite{MMS} they use $\D_{s^{-1}_{out}}\simeq \K^{-1}$ to twist,
which is equivalent since the categories of the twisted PROP(erad)s are equivalent by Proposition \ref{twistprop}
\end{rmk}

\subsubsection{Odd operads and anti-cyclic operads as twisted operads and their relation to $\K$.}
\label{oddkpar}
 Now we can make the mantra (1) precise by using $\K$ twisted instead of odd.
\begin{thm}
\label{kversionthm}
All $\K$ twisted versions in Table \ref{oddtable} carry a natural odd Lie bracket on the direct sum of their coinvariants.
Their shifts accordingly carry a Lie bracket.
\end{thm}

\begin{proof}
The first statement is just a rephrasing of our previous results, using Proposition \ref{twistprop} and Lemma \ref{twistisolem}
 except for the case of $\K$--modular operads which
for the bracket reduces to the case of odd cyclic, since the gluing is only along trees.
\end{proof}

\section{Odd self-gluing and the BV differential}
\label{bvpar}
In this paragraph, we deal with mantra (3). For this we need odd self-gluings.
We have
already treated odd wheeled PROP(erads).
We now turn to $\K$--modular operads.

The most important fact that we need is that $\K$--modular operads have an odd self--gluing structure
that is the operations $\bullet_{ss'}:\O(S)\to \O(S\setminus \{s,s'\})$
such that for four element subsets  $\{s,s',t,t'\}\subset S$ and $a\in \O(S)$
\begin{equation}
\label{anticomeq}
\bullet_{ss'}\bullet_{tt'}(a)=-\bullet_{tt'}\bullet_{ss'}(a)\in \O(S\setminus \{s,s',t,t'\})
\end{equation}

Using the language of graphs, the two different operations correspond to a graph with one vertex, with flags indexed by $S$ and with
two pairs of flags $\{s,s'\}$ and $\{t,t'\}$ joined together as edges $e_1$ and $e_2$, the two compositions however correspond to
$\circ_{e_1\wedge e_2,\Gamma}$ and $\circ_{e_2\wedge e_1,\Gamma}$ in the notation of  Lemma \ref{twistedcomplem},
which differ by a minus sign.

\begin{prop}
The operator $\Delta$ defined on each $\O(g,S)$ defined by
\begin{equation}
\Delta(a)=\sum_{\{s,s'\}\in S, s\neq s'}\bullet_{ss'}(a) \in \bigoplus_
{\{s,s'\}\in S, s\neq s'}\O(g+1,S\setminus \{s,s'\})
\end{equation}
satisfies $\Delta^2(a)=0$ for any $a\in \O(g,S)$.
\end{prop}

\begin{proof}
We consider the component $S\setminus \{s,s',t,t'\}$ for fixed $s,s',t,t'$.
It will get six contributions which appear pairwise. Each pair
corresponds to an ordered partition $\{a,b\}\amalg \{c,d\}$ of $\{s,s,t,t'\}$
and the two terms appear with opposite sign. These are the compositions for the $S\setminus \{s,s',t,t'\}$--labeled graph with one vertex and
two edges in both orders of the two edges.
\end{proof}

\begin{rmk}
Here we chose to index by two element subsets of $S$. If we index by tuples $(s,s')$ and we are in characteristic different from two then we obtain
the more familiar form:
$$
\Delta(a)=\frac{1}{2}\sum_{(s,s')\in S, s\neq s'}\bullet_{ss'}(a) \in \bigoplus_
{\{s,s'\}\in S\times S, s\neq s'}\O(g+1,S\setminus \{s,s'\})
$$
\end{rmk}

Passing to coinvariants, we obtain an instance of mantra (3)
\begin{prop}
$\Delta$ induces a differential on $\OSp$ that is $\Delta^2=0$.
This differential lifts to the cyclic invariants and to the biased setting.
\end{prop}

\begin{proof}
On $\OSp$ the equality follows directly from (\ref{anticomeq}). For the lifts, we remark that $\{0,\dots \hat i, \dots \hat j, \dots n\}$
has a natural cyclic and linear order.
\end{proof}

\begin{rmk}
In the biased setting as shown in \cite{Zwie,Schwarz} it is sufficient to lift $\Delta$ to $\bullet_{n-1n}$ on $\O(n)$.
\end{rmk}

Now we have mantra (2) in the form:
\begin{thm} \label{r5}
The $\K$ twisted version of modular operads, wheeled PROP(erad)s and the chain level Schwarz extended modular operads (EMOs)
carry a differential $\Delta$ on their coinvariants.
\end{thm}

Where the EMOs are discussed in \ref{emopar}.

\section{Multiplication, Gerstenhaber and BV}\label{gbvsec}

So far for cyclic and modular operads, we
 have only constructed (odd) Lie brackets and differentials.
In order to upgrade them to Gerstenhaber respectively Poisson algebras
and BV operators, we need an additional multiplicative structure.

Following \cite{Zwie,Schwarz,HVZ}
we show that there is a natural external multiplication one can introduce by
going to disconnected graphs. It is the external
multiplication that is natural to consider in the master equation as that
equation is a linearization of an equation involving an exponential.

There is a second type of multiplicative structure that is possible.
 This is an internal product; that
is an element $\mu\in \O(2)$ which is associative.
Although a little bit outside the main focus of the paper,
we deal with the second type of multiplication in order to
contrast it with the one above.
This second type of structure appears in Deligne's
conjecture \cite{KS,McSm,Vor,del,BF,T}, its cyclic generalization \cite{cyclic}.

A last possibility is an $A_{\infty}$ version which was
studied in \cite{TZ,KSchw,manfest,Ward},
but that goes beyond the scope of this paper.  Here one relies on the fact that the $A_\infty$-operad represents the functor assigning MC elements to an operad via the above Lie algebra construction.  The above results, along with the results of section $\ref{feynmanpar}$, then give a suitable framework for generalizing this internal multiplication outside the operad case, see \cite{Ward2}.

\subsection{Non-connected versions.}
\label{NCpar}
{\it  A priori} an operad of the above kinds has no multiplication.
We can however add a generic one, by passing from connected graphs to
non-connected ones.  In general, to get the nc-version one uses compositions along graphs of the same type as before, but drops the assumption that
the relevant graphs are connected.
Some care must be used however, since it is not always clear how this
should be implemented.

\subsubsection{Non-connected (odd) operads}
It turns out that operads are the most difficult example from this perspective and there are several nc-generalizations.
This is because the straightforward way of treating the graphs along which the compositions
are defined needs to be interpreted. Namely  taking disjoint
unions of rooted trees one arrives at rooted forests.  This changes the number of outputs from strictly $1$ to any number $m$; the number of trees in the forest.  There are at least 4 ways to deal with this:
\begin{enumerate}
\item The PROP generated by an operad
\item Nc-dioperad generated by an operad

\item The free nc--version according to \cite{feynman}.
\item The $B_+$ construction or operads with nc  multiplication.
\end{enumerate}
For a given operad $\O$ there is always a free extension $\O^{nc}$ yielding an object of the given class.

The basic idea for all these versions is that we move to a collection of rooted corollas as ``vertices''. If we simply use mergers,
then we see that merging rooted corollas, we obtain directed corollas. The picture one should have in mind is a box which contains
the $m$ corollas, see Figure \ref{boxedoperadfig}.

\begin{figure}
\centerline{\includegraphics[height=.5in]{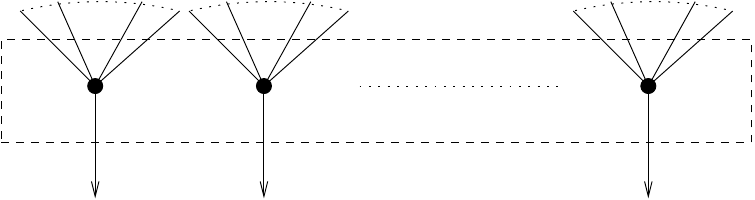}\hfill \includegraphics[height=1in]{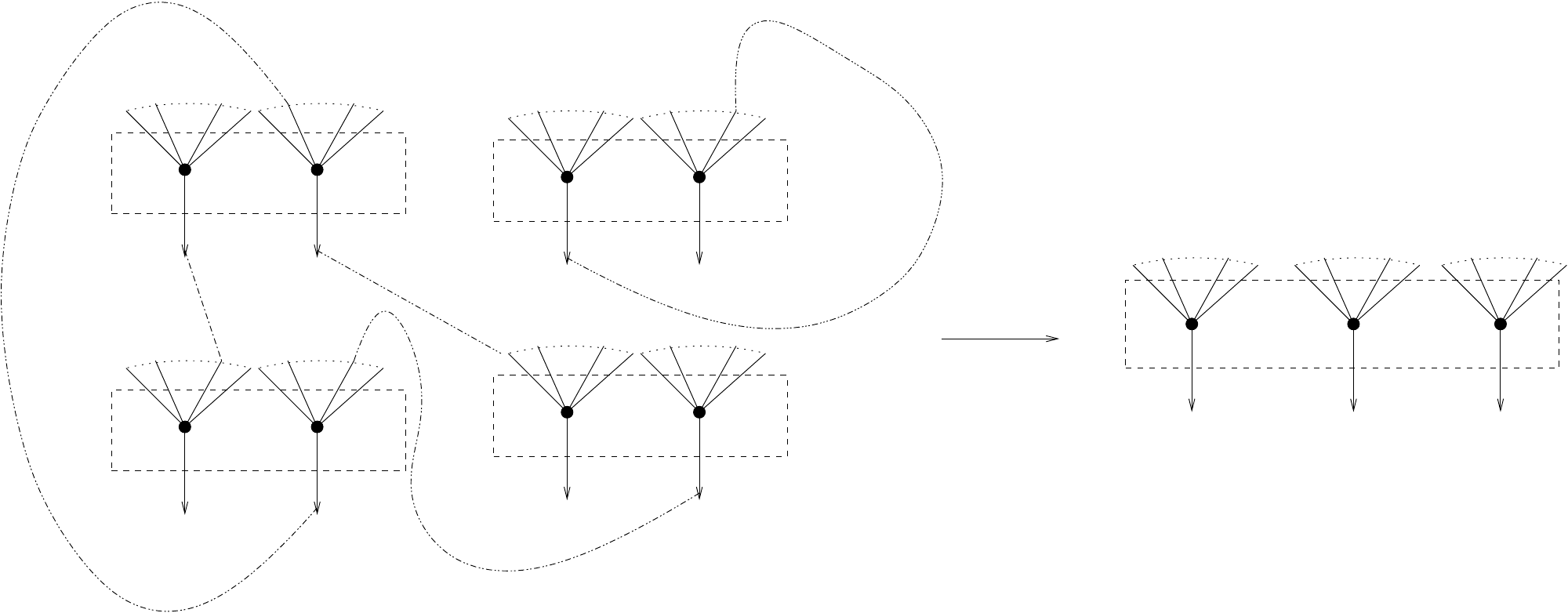}}
\caption{\label{boxedoperadfig}Boxing $m$ corollas and  connecting boxed corollas with trees and \label{freenccompfig} contracting their edges.}
\end{figure}

(1) Thinking of a directed corolla with $m$ outputs as any collection of $m$ corollas with the correct number of inputs,
and allowing PROP gluings between these collections we arrive at PROPs generated by an operad.
That is we have $\SS_n\times \SS_m$ modules $\O(n,m)$ with a decomposition given by equation
(\ref{oppropeq}) and allow PROP operations between these. This means that we use the fine structure of the box for the decomposition,
but for the compositions only use the outside structure of the boxes.
The nc--extension of an operad $\O$ is Example \ref{propgenopex}.

(2)  We can proceed as in (1) but restrict to only the dioperadic gluings.

(3) The free nc--construction of \cite{feynman} yields the following concrete realization: The gluings between the boxed corollas are defined by first removing all the boxes, then performing all possible gluings, which are rooted forests and finally reboxing the result,
see Figure \ref{freenccompfig} for an example.

Here we have $\SS_{n_1}\times \dots \times \SS_{n_m}$--modules $\O(n_1,\dots,n_m)$ with compositions given by rooted forests. The operations are  generated by single edge gluings, and disjoint union. The dioperadic gluings are a subset.
The nc--extension of a given $\O$ is $\O^{nc}(n_1\dots,n_m):=\bigotimes_{i=1}^m\O(n_i)$ together with $\boxminus=\otimes$ and the
operations induced by the $\circ_i$.

In all the above cases the triple is given by inserting into the boxes or dually expanding the boxes to graphs.

(4) Finally, we can gather the trees in the rooted forest together by adding a common root vertex. This is the $B_+$ operator of \cite{CK}, see also \cite{feynman,GKT}.
It is depicted in Figure \ref{bplusfig}.

\begin{figure}
\centerline{\includegraphics[height=1in]{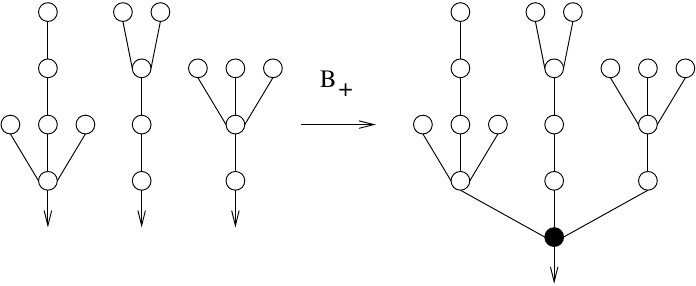}\hfill  \includegraphics[height=1in]{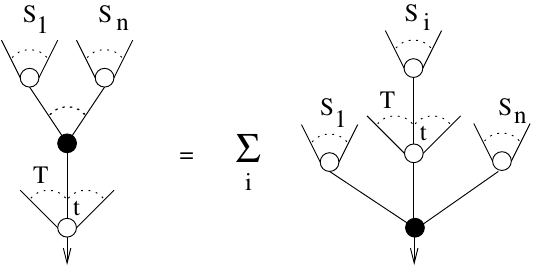}}
\caption{\label{bplusfig} The operator $B_+$ on a forest of three rooted trees and \label{ncbvmultfig}   the relation for nc--multiplication}
\end{figure}

Here for the triple, we need to insert into a vertex {\em and} then sum over all possible connections. One way to implement this is to consider the triple of b/w bipartite rooted trees with insertion into the white vertices (see e.g.\ \cite{del})
and then to force the relation of Figure \ref{ncbvmultfig} at each black vertex by moding out the respective ideal in the free Abelian group generated
by the trees.

Now an algebra is just a collection of $\SS_n$ modules  $\O(n)$ together with the $\circ_i$ and a horizontal composition $\boxminus$ that satisfies equations

\begin{eqnarray}
\label{ncmulteq}
a\circ_i (b\boxminus c)&=&(a\circ_i c) \boxminus b +a\boxminus (b\circ_ic) \nn\\
 (a\boxminus b)\circ_i c&=&\begin{cases} (a \circ_i b ) \boxminus c& \text{ if  index $i$ belongs to $a$}\\
 (a\circ_ic) \boxminus b &\text{ if the index $i$ belongs to or $b$.}\end{cases}
 \end{eqnarray}

The free extension is given by $\O^{nc}(n):=\bigoplus_{(n_1,\dots,n_m):\sum n_i=n}\bigotimes \O(n_i)$. Again $\boxminus=\otimes$ and
the $\circ_i$ are induced by the original ones by summing over all roots, see the formula (\ref{rootsumeq}) below.

{\sc The relations between the constructions are:} that (2) embeds into (1) and (1) into (3), they just have more gluing operations.
(4) embeds into (3) via

\begin{equation}
\label{rootsumeq}
a\circ_ib=\sum_{r=1}^m a\ccirc{i}{r}b \text{ if } b\in \O(n,m)
\end{equation}

\begin{rmk}
Operads with such an nc--multiplication arise for instance from operads with associative multiplication via Gerstenhaber's construction, see \S\ref{intmultpar}.
\end{rmk}

{\sc Odd versions:}
In order to achieve the correct odd notion, we again have to twist the relevant triple.
The twist is by $\K$ which as previously is the determinant of the edges of the graph describing the decomposition. We call an algebra over such
a triple a non--connected odd operad. Notice  that $\Gbracket$ is well defined
as the sum over the non--self gluings.

\begin{thm}\label{r6}
Given a non--connected odd operad in any of the four versions above,
the odd Lie bracket $\Gbracket$ is Gerstenhaber with respect to $\boxminus$.
\end{thm}
\begin{proof}
This just boils down to the fact that before anti--symmetrizing on the left hand side of (\ref{dereq}),
we have a summand corresponding to connecting the inputs/output of $a$ to any element
of the set $S\amalg T$ if $b\in \O(S)$ and $c\in \O(T)$ say. The ones connecting the root
to $S$
are the first term, while
the ones connecting the root to $T$ are the second term of the rhs.
For the cases (1),(2),(3)
this follows from equation (\ref{nceq}) and for (4) by definition, i.e.\ equation (\ref{ncmulteq}).
\end{proof}

\subsubsection{Nc-cyclic}
For cyclic operads and modular operad the non--connected notions
have not appeared in the literature yet --- as far as we are aware.
The relevant triples
are those of forests (collections of trees).
For the triple, we insert forests into vertices. Notice that since there is no direction on the flags,
this operation is well--defined unlike the operad case.

 We will call the algebras over these triples nc--cyclic operads.
Again the relevant morphisms are given by isomorphisms, contracting
edges and combining collections. The disjoint union of two one vertex graphs stands for a merger and
gives a horizontal composition $\boxminus:\O(S)\otimes \O(T)\to \O(S\amalg T)$.
The twist by $\K$ makes sense and we obtain the notion of odd--nc--cyclic operad.

\begin{thm}\label{r7}
Given an  odd nc-cyclic operad,
the odd Lie bracket $\cGbracket$ is Gerstenhaber with respect to $\boxminus$ on the coinvariants
$\Op_{\SS}$.
\end{thm}
\begin{proof}
This just boils down to the fact that on the left hand side of (\ref{dereq}),
we have a summand corresponding to connecting the root of $c$ to any element
of the set $S\amalg T$. The ones connecting to $S$ are the first term, while
the ones connecting to $T$ are the second term of the rhs.
\end{proof}

\subsubsection{Nc-modular operads}
\label{ncmodularsec}
For nc--modular operads the basic underlying triple will be non--connected
graphs. We must however deal with the genus labeling. Since the graphs
are not connected one should replace $g$ by $\-\chi$ where $\chi$ is the Euler characteristic.
 For any graph, its Euler characteristic is given by
the Euler characteristic of its realization.
Viewing it as a 1--dimensional CW complex and contracting any tails,
we get that
$$\chi(\Gamma)=b_0(|\Gamma|)-b_1(|\Gamma|)=|\text{vertices of $G$}|-|\text{internal edges
of $\Gamma$}|; $$

If $\Gamma$ is connected then $1-\chi(\Gamma)=g$.

We replace the genus labeling by the labeling by $\gamma$.
That is a function $\gamma:$ vertices of $\Gamma\to \N$.

The total $\gamma$ is now
$$\gamma(\Gamma)=1-\chi(\Gamma)+\sum_{v \text{ vertex of $\Gamma$}}\gamma(v)$$
This means we get non-self gluings $\scirct$ for which $\gamma$ is again additive
and self-gluings $\circ_{ss'}$
increasing $\gamma$ by one. There is also
the collecting together which gives a horizontal map
$\boxminus:\O(\gamma, S)\otimes
\O(\gamma',T)\to \O(\gamma+\gamma',S\amalg T)$.

The triple is now given as usual. Just as in the modular case, the multiplication in the triple
expands the vertices into graphs of the corresponding type $(Flags(v),\gamma(v))$.  The twist by $\K$ makes sense and we obtain the notion of an nc-$\K$-modular operad.  Again $\Delta$ is well defined as the sum over all self-gluings.

\begin{thm}\label{ncthm}
For an nc-$\K$-modular operad $\O$,
the sum over non-self gluings gives an odd Lie bracket $\cGbracket$ on the coinvariants (both cyclic and full)
which is Gerstenhaber for the horizontal multiplication on $\Op_{\SS}$.  The differential $\Delta$ is a BV operator for the horizontal multiplication on $\Op_{\SS}$ and its Gerstenhaber bracket is the bracket induced by $\cGbracket$.
\end{thm}
\begin{proof}
The proof can either be done by direct calculation or by
the following argument which is essentially an adaption of that of \cite{HVZ}.
If we look at the equation (\ref{poissoneq}) then taking $\Delta(a\boxminus b)$
decomposes into three terms. All self--gluings of $a$, all self--gluings of $b$
and all non-self gluings between $a$ and $b$, which if one is careful
with the signs give all the gluings.
A pictorial representation is given in Figure \ref{bvcalcfig}.
Again one has to be careful that one uses  coinvariants,
which is where $\cGbracket$ satisfies the Jacobi identity.
\end{proof}

\begin{figure}
\includegraphics[width=\textwidth]{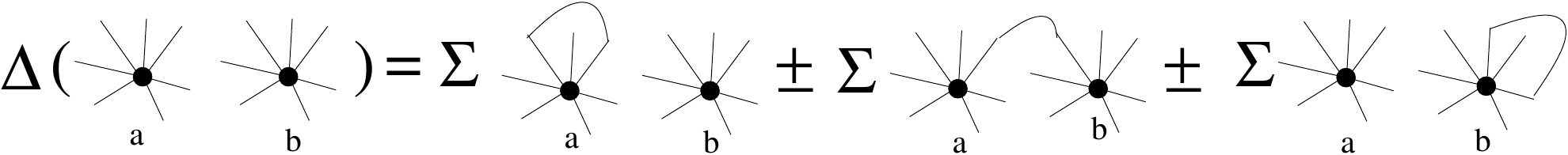}
\caption{\label{bvcalcfig}The three terms for checking the BV property}
\end{figure}

\subsubsection{Nc-extension}
Just like there is the PROP generated by an operad, a cyclic or  (twisted) modular operad
generates an nc--version.  Here the operation $\boxminus$ is just taken to be $\otimes$ and one sets
\begin{equation}\O^{nc}((\gamma,n))=\bigoplus_k\bigotimes_{\begin{tabular}{ll}$(n_1,\dots, n_k):\sum_i n_i=n$\\
$(g_1,\dots, g_k):\sum 1-g_i=\gamma$\end{tabular}}\O((n_i,g_i))
\end{equation}

\subsection{$\K$-twisted Realization of the Mantra}
We can now formulate mantra (3) in this context.

\begin{thm}
For the nc-versions of odd cyclic operads and $\K$--twisted modular operads as
well as for $\K$--twisted version of wheeled PROPs
the operator $\Delta$ is a BV operator on the
coinvariants which induces the previously constructed Gerstenhaber bracket.
\end{thm}

\subsection{Non-free nc-extensions: (Cyclic) Operads with multiplication}
\label{intmultpar}
There are basically two ways to get an nc--multiplication for (cyclic) operads, the first is the free one, which we discussed above and the second is
using an internal multiplication which we discuss now.

Let $\mu\in \O(2)$, s.t.\ $\mu\circ_1\mu=\mu\circ_2\mu$.
An operad together with such an element is called an
operad with multiplication.
Indeed on $\Op$, $\mu$ defines a graded associative multiplication via
$a\otimes b\mapsto (\mu \circ_2b)\circ_1a$.

Such an element also gives rise to a differential $da:=\{a\bullet \mu\}$. The following theorem can be extracted from \cite{Gerst}, see e.g.\
\cite{woods}.

\begin{thm}
For an operad with multiplication the odd bracket $\Gbracket$ is odd Poisson, aka.\ Gerstenhaber, up to homotopy; that is the equations hold up to $im (d)$.
\end{thm}

Indeed the required equation for $\mu$ to be an nc--multiplication is (\ref{ncmulteq}).

In the cyclic situation for an operad with a unit $1\in \O(0)$ for $\mu$,
one can define degeneracy maps via $s_i(a):=a\circ_i 1$.  Then one can define the operator $B=s(1-t)N$ on the complex $\Op$ with
the differential $d$ (or the sum of the internal differential and $d$).
On  the reduced
complex is just $sN$. The calculation in \cite{cyclic}
shows that
\begin{thm}
For a cyclic operad,
$B$ is a differential on the reduced complex and descends to a
BV operator for $\mu$ on the cohomology.
Moreover the induced bracket agrees
with the one coming from the Gerstenhaber structure.
\end{thm}

This type of BV operator is internal and has {\it a priori} nothing
to do with the external
$\Delta$ we considered above. They also yield different
Gerstenhaber brackets, namely $\Gbracket$ and $\cGbracket$.

Thus taking coinvariants, they are related {\it a posteriori}.
Moreover if $\mu$ is cyclic, then the gluing can be thought of
as composing both elements with $\mu$ and putting in a co--unit.
The precise relationship and interplay between the two BV formalisms is an
interesting open problem.

\section{(Co)bar constructions, the Feynman transform and the master equation}\label{feynmanpar}

The bar-cobar constructions are used to give cofibrant (quasi-free) resolutions of operads\cite{GiK}.  These constructions have been generalized e.g. to cyclic operads \cite{GeK1}, dioperads \cite{Gan}, properads \cite{Val}, wheeled properads \cite{MMS}.  On the other hand the notion of the Feynman transform plays a similar role in the category of modular operads, \cite{GK}.  Here one encounters the added complication of passing to the $\fr{K}$-twisted triple; consequently the Feynman transform of a modular operad is no longer a modular operad.

With hindsight, and with the above work, one sees that the Feynman transform is the more general, unifying notion.  The discrepency in the level of complication arises from the fact that in the various contexts the twist $\fr{K}$ may or may not arise from a coboundary.  In this section we define a general Feynman transform which captures the above constructions.  Roughly speaking, an algebra over a triple is replaced by a quasi-free algebra over the $\fr{K}$ twisted triple.  The failure of freeness is exactly measured in its various guises
by the master equations associated to the Lie, Gerstenhaber, and BV operations constructed above.

In this section we work in the category $\dgVect$.

\subsection{(Co)bar construction aka Feynman transform}\label{mepar}  For this discussion we fix a triple $\T$ encoding one of the operad-like structures above. Examples are: operads, cyclic operads, their non-$\Sigma$ variants, modular operads, dioperads, wheeled properads, wheeled props, also their twisted versions.  The Feynman transform is a functor
\begin{equation*}
\F\colon \T\text{-algebras} \to \T_\fr{K}\text{-algebras}
\end{equation*}
defined by
\begin{equation*}
\F(\O):=(F_{\fr{K}}(\fr{s}\op{O}^* ), d\F)
\end{equation*}
where $\op{O}^*$ is the $\SS$-module given by the linear dual of $G(\O)$ and where $F_{\fr{K}}$ is the free algebra over $\T_{\fr{K}}$.

Taking linear duals gives a differential that is dual to the composition given by contracting an edge.  More precisely the total differential $d\F$ on $\F(O)$ is the sum $d\F := \del_{\O^*} +\del$ where $\del_{\O^*}$ is the internal
differential induced from the differential on the $\O$, and $\del$ is a new external differential
whose value on the term $(\K(\Gamma)\otimes \O^*(\Gamma))_{Aut(\Gamma)}$ is
given as follows. Consider  $\hat \Gamma$ together with an edge $e$ such that $\hat\Gamma/e\simeq \Gamma$.
Then  there is a map $\circ_e: \O(\hat \Gamma)\to \O(\Gamma)$ which
composes along $e$. Since $\O$ is an algebra over $\T$ for such a pair there is a map
\begin{equation}
\del_{\hat G,e}: \K(\Gamma)\otimes \O^*(\Gamma)
\stackrel{\eps_e\otimes \circ_e^*}{\longrightarrow} \K(\hat \Gamma)\otimes \O^*(\hat \Gamma)\\
\end{equation}
where $\eps_e$ is the multiplication by the  basis element $[e]$ of $Det(\{e\})$.
Now the matrix element $\del$ between $ (\K(\hat\Gamma)\otimes \O^*(\hat\Gamma))_{Aut(\hat\Gamma)}$
and  $(\K(\Gamma)\otimes \O^*(\Gamma))_{Aut(\Gamma)}$
is the sum over all $\del_{\hat \Gamma,e}$ for which $\hat \Gamma/e\simeq \Gamma$.
If there is no such edge, then the matrix element is zero.

The reason to introduce the twist by $\K$ into the picture is to make $\del$ into a differential.
Indeed applying it twice inserts two edges in all possible ways and each term appears twice:
once with each possible ordering of the two edges. Due to the presence of the tensor factor
$Det(Edges)$ these terms differ by a minus sign and cancel.

One may observe that this definition of Feynman transform agrees with the bar-cobar constructions mentioned above, modulo coboundaries (and modulo taking linear duals depending on one's conventions).

\begin{rmk}  Since $\fr{K}^{\tensor 2}$ is a coboundary, we may consider double iteration of the Feynman transform to be an endofunctor.  The result of this endofunctor is always quasi-free, and is often a resolution.  However, in the cases allowing non-connected graphs we do not resolve the horizontal composition and hence can not expect a resolution.
\end{rmk}

\subsection{Algebras over the Feynman transform and Master Equations.}
One has to distinguish: as a graded object $\F(\O)$ is free but as a {\it differential} graded object it is not.  As mentioned above this discrepency is captured by a so-called master equation  in the various contexts.  These master equations use the brackets and BV operators defined above.  We will now briefly recall the relevant operations and give the master equations in the various contexts.  Before doing so we mention two technical points.  First the solutions to the master equation live in the direct product not the direct sum, i.e. we can have infinitely many non-zero components.  To this end we define $\oprod$ to be the direct product of the respective spaces of coinvariants.  It is straight-forward to check that the algebraic operations defined {\it a priori} on the direct sum extend to the product (see \cite{Ward2}).  Second, the degree of all master equation solutions is even, but the exact degree depends on conventions for differentials and brackets.  To be concrete, let us assume that our brackets and differentials have degree $-1$, and thus we stipulate that all master equation solutions have degree $0$.  With that the master equations are:

{\bf Operads}.  Let $\op{O}$ be a dg odd operad.  Then $\oprod$ is an odd pre-Lie algebra, where the operation $\circ$ was defined in Section $\ref{opsec}$.  An element $S\in \oprod$ is a solution to the master equation if
\begin{equation*}
\partial_\op{O}(S)+S\circ S=0
\end{equation*}

{\bf Cyclic operads and dioperads.}  Let $\op{O}$ be a dg odd cyclic or di- operad .  Then $\oprod$ is an odd Lie algebra, with bracket as defined in sections $\ref{cycsec}$ and $\ref{disec}$.  An element $S\in \oprod$ is a solution to the master equation if
\begin{equation*}
\partial_\op{O}(S)+\frac{1}{2}\{S, S\}=0
\end{equation*}

{\bf Wheeled properads and modular operads.}  Let $\op{O}$ be a dg odd wheeled properad or modular operads.  Then $(\oprod, \{-,-\}, \Delta)$ is a dg Lie algebra, where the operations were defined in sections $\ref{disec}$ and $\ref{gbvsec}$.    An element $S\in \oprod$ is a solution to the master equation if
\begin{equation*}
\partial_\op{O}(S)+\frac{1}{2}\{S, S\}+\Delta(S)=0
\end{equation*}

Here in the wheeled prop case $\Delta$ is a BV operator; where as in the connected cases there is no multiplication.  In general, in each context the master equation is the same in the connected and nc versions, see $\ref{ncgensec}$.

\subsubsection{Realization of Mantra (4).}  Now we let $\P$ be an algebra over a triple $\T$ enconding one of the above structures and let $\O$ be an algebra over the twisted triple $\T_\fr{K}$.  Recall that this implies $\O\tensor \P$ is then also an algebra over $\T_\fr{K}$.  In all cases we have:

\begin{thm}\label{r8}
There is a natural bijective corresondence between $Hom_{dg}(\F(\O),\P)$ and solutions to the master equation in $(\O\otimes \P)^{_{\prod}}_{\SS}$.
\end{thm}

For operads this result is classical; see e.g. \cite{MSS} for a discussion.  For modular operads this result is due to Barannikov \cite{Ba}.  For wheeled properads an example of this result is given as Theorem 3.4.3 of \cite{MMS} (for the case $\O$ having the ground field in each bi-arity).  For cyclic operads a version of this result is used in \cite{Ward2}.  There are two ways to prove this theorem.  The first is to prove each case individually, whereas the second is to build a framework general enough so that the series of statements that comprise this theorem becomes a single statement.  The latter is done in our subsequent work \cite{feynman}.  In either case the proof is essentially an unraveling of definitions.

\begin{rmk}  A particularly important case of the above theorem is when $\P$ is the endomorphism operad, so that such master equation solutions parameterize $\F(\O)$-algebra structures.  This is particularly relevant of $\O$ is Koszul in which case $\F(\O)$-algebra structures are equivalent to strongly homotopy algebras structures over the Koszul dual.
\end{rmk}

\begin{rmk}
Including properads in the above list would essentially recover the dioperadic construction of \cite{Gan} and not the resolution of \cite{Val} (which uses a suitably altered bracket). We defer that discussion to \cite{feynman}, where we introduce transforms depending on a fixed set of generators.
\end{rmk}

\subsubsection{NC-generalization}\label{ncgensec}
In the nc extension of the above situation the master equations remain the same.  For example in the case of modular operads such solutions are also exactly the solutions of
\begin{equation}\label{exponentialeq}
(d+\lambda\Delta )e^S=0
\end{equation}
Here the exponential is formal for the product given by $\boxminus$.
This is in accordance with quantum field theory, where the exponential
gives the sum over all not--necessarily connected Feynman graphs.

For the quantum master equation, we never want to resolve the horizontal composition. This operation
yields the multiplication for the Gerstenhaber/BV structure and is inherent in the definition of $e^S$ which is the physically
relevant exponentiated action \cite{Zwie,KontSchw}.

\section{Geometric examples}\label{toppar}
In this section we give some geometric examples which lead to occurrences
of mantra (5).
There are basically two kinds: open and closed.
These are motivated by the constructions of \cite{HVZ} and
\cite{KSV}, and ultimately by \cite{Zwie}.
Informally speaking the common feature of the following closed
examples is an $S^1$--action on
the outputs, which can be transferred to a twist gluing.
Such a twist gluing will be an $S^1$ family.
Passing to homology or chains this 1--parameter family gives  degree $1$ to the gluing making the gluing odd.

The other type of gluing is a gluing at boundary punctures. In order
for it to be odd one must consider orientations and for it to get
degree one, one has to pick a grading by codimension as we explain
below. The paradigm for this is contained in \cite{HVZ}, but
was previously also inherently present in Stasheff's associahedra
and more recently in \cite{KSchw} for the Gerstenhaber operad.

\subsection{Topological $\Sn\wr S^1$ modular operads}

Suppose we have a topological modular operad $\CO$. We also assume that $\CO((g,n))$ has an $(S^1)^{\times n}$ action which together with the $\Sn$ action gives an action of $\Sn\wr S^1$.
For $\phi\in S^1=\R/\Z$ let $\rho_i(\phi)a= (0,\dots,0,\phi,0,\dots)(a)$ where the non--zero entry is in the $i$--th place.

\begin{df}
A topological $S^1$--modular operad is  a modular operad $\CO$ with an $\SS\wr S^1$ action that is {\em balanced} which means
that
\begin{equation}
\rho_i(\phi)(a)\ccirc{i}{j} b = a \ccirc{i}{j}\rho_j(-\phi)(b) \text{ and } \circ_i^j \rho_i(\phi(a))= \circ_i^j (\rho_j(-\phi( a))
\end{equation}

Likewise we define the $S^1$--twisted versions of (cyclic) (twisted) operads and (wheeled) (twisted) PROP(erads) or also di--operads, etc.
\end{df}

\begin{nota}
To shorten the statements, we will call any $\O$ belonging to any of the categories in the previous sentence of {\em composition type}.
\end{nota}

\begin{df} The twist gluing $\Sccirc{i}{j}$ of $a$ and $b$ is the $S^1$ family given by
$\rho_i(S^1)a\ccirc{i}{j}b$
\end{df}

This type of twist gluing does not give a nice operad type structure on the topological level, unless as suggested by Voronov, one uses
the category of suitable spaces with correspondences as morphisms. It does however give nice operations on singular chains
and hence on homology.

Namely, given two chains $\alpha\in S_k(\O(n))$ and $\beta\in S_l(\O(m))$ we define the chains
\begin{equation}
\alpha\cbullet{i}{j} \beta:=S_*(\ccirc{i}{j})EZ\, S_*( id\times \rho_j)  (\alpha \times \rho_i  \times \beta)
\end{equation}
as chains parameterized over $\Delta^k\times \Delta^1\times \Delta^l$ pushed forward with $\rho_j$ and the Eilenberg Zilber map
to give a chain in $S_{k+l+1}(\O(n)\times S^1\times\O(m))$. Here $\Delta^1$ maps
to the fundamental class $[S^1]$.
Likewise we define
\begin{equation}
\bullet_{ij}\alpha:=S_*(\circ_{ij})S_*(\rho_{i})([S^1] \times \alpha)
\end{equation}

This type of operation of course generalizes and restricts to all $\O$ of composition type.

\begin{thm}
The chain and homology  of any  $S^1$--twisted  $\O$ of composition type are  $\K$--twisted versions of that type.
\end{thm}
\begin{proof}
We see that the compositions are along the graphs of the triple, where the edges are now decorated by the fundamental class of $S^1$.
This lives in degree $1$ and hence the compositions get degree $+1$. If we now shift the source of the morphisms by $-1$
we get operations of degree $0$ and hence we get composition morphisms for the $\K$ twist of $\O(\Gamma)$. \end{proof}

\subsubsection{New examples and applications: $\Arc$, framed little discs and string topology}

One example is given by  the $\mathcal {A}rc$ operad of \cite{KLP}, which has such a balanced $S^1$ action. The twist gluing and BV operator are discussed
in \cite{archbk}.
 The $\Arc$ operad  contains the well known operad of framed
little discs \cite{cact} which is  a cyclic $S^1$ operad.

A rigorous topological  version of the Sullivan PROP was given
 in \cite{hoch1}
this structure is actually a quasi--PROP which is only
associative up to homotopy, but it has a cellular PROP chain model.
Just like in the $\Arc$ operad there is an action of $S^1$ on the inputs,
as these are fixed to have arcs incident to them.
Thus we can twist glue by gluing in the $S^1$ families.

\subsubsection{Co--invariants }
Given an $S^1$--twisted $\O$ of composition type, we can consider its $S^1$--coinvariants.
For concreteness we will treat modular operads, the other types work analogously.
Here $\O_{S^1}((g,n)):=\O((g,n))_{(S^1)^{\times n}}$.
Let $[\;]:\O\to \O_{S^1}$ denote the projection.

Then the twist gluings provide a natural family of gluings on the coinvariants:
Namely if $[\alpha]$ and $[\beta]$ are two classes in the coinvariants,
we can set
\begin{equation}[\alpha]\cbullet{i}{j}[\beta]:=[\alpha \ccirc{i}{j}^{S^1}\beta] \quad \bullet_{ij}([\alpha]):=[\circ^{S_1}_{ij}\alpha]
\end{equation}
\begin{prop}
These operations are well defined and furnish a $\K$ twisted composition structure on the chain and homology level.
\end{prop}

\begin{proof}
The fact that this is well defined follows from the fact that the action is balanced. The second part is as above.
\end{proof}

\begin{rmk}
The co--invariants of the Sullivan PROP  are also what gives
 rise to an $L_{\infty}$ structure \cite{CS}, which seems to be true in general.
\end{rmk}
\subsubsection{Scharz's modular and extended modular operads}
\label{emopar}
There are other early examples like the Schwarz--modular operads MO \cite{Schwarz} where there are only self--gluings and
a horizontal composition. In order to get an odd operation on the chain level Schwarz considers so called EMOs (extended modular operads),
these carry just as above an $S^1$ action which gives an $\SS_n\wr S^1$ action on each $\O((n))$.

\subsection{The paradigm: Real blow-ups and the Master equation}

\subsubsection{Closed version}
A particularly interesting type of situation occurs if one augments an operad with an $S^1$ action.
The prototype for this is the collection $\overline{M}_{g,n}^{KSV}$ of real blow ups of the Deligne--Mumford spaces along their
compactification divisors as defined in \cite{KSV}.

Here, before the blow--up, the spaces $\overline M_{g,n}$ form a modular operad --- even the archetypical one. The gluing of
two curves is given by identifying the marked points and producing a node.
One feature of the compactification is that the compactification divisor is composed of operadic
compositions. More precisely for each genus labeled graph $\Gamma$ of type $((g,n))$ there is a map $\overline {M}(\Gamma)\to \Mgn$
where $\overline {M}(\Gamma)=\times_{v\in V(\Gamma)}M_{(g(v),Flag(v))}$ and in
particular the one-edge trees define a normal crossing divisor.

Now after blowing up, the spaces $\Mksv$ do not form a modular operad anymore, since one has to specify a vector
over the new node. This is the origin of the twist gluing.
One could have also added tangent vectors at each marked point and the nodes. This would give a modular operad. The KSV--construction is then just the twist gluing on the co--invariants.

The master equation now plays the following role.
Let $S=\sum_{g,n}[\overline{M}_{g,n}^{KSV}/\Sn]$,
where one sums over fundamental classes in a suitable sense. One such framework
is given in \cite{HVZ} where geometric chains of Joyce \cite{joyce}
are used.

The boundary in this case is essentially
 the geometric boundary of the fundamental class
viewed as an orbifold with corners. Notice that while in the DM setting
the compactification was with a divisor i.e. of complex codimension one, after
blowing up in the KSV setting the compactification is done by a real codimension
one bordification. Thus $dS$ is the sum over these boundaries, which
are exactly given by the blow ups of the divisors and these
correspond exactly to the surfaces with one double point, either self glued
or non--self glued.
Working this out one finds that  $S$ satisfies the master equation.

\subsubsection{Open gluing case/orientation version}

Likewise there is a construction in the open/closed case in \cite{HVZ}.
Here the relevant moduli spaces are the real blow--ups
$\overline{\mathcal{M}}_{g,n}^{KSV\, b,\vec{m}}$ of the moduli space
$\overline{\mathcal{M}}_{g,n}^{b,\vec{m}}$ introduced in \cite{Liu}.
These are the moduli spaces of genus $g$ curves with $n$ marked labeled points,
$b$ boundary components and $\vec{m}$ marked labeled points on the boundary.
 In the closed case the blow up inherits an orientation because before compactifying the moduli space has a natural complex structure. In the open/closed case one can define iteratively
 the orientation by lifting or pushing the natural orientation of $M^{HVZ b,(1,...,1)}_{g,n}$ (see \cite{IS}) along fiber bundles that at the end reach any open/closed moduli space.

Whereas the degree $1$ in the closed case came from the fundamental class,
here the grading comes from a grading by codimension
in the corresponding moduli space. This agrees with the geometric dimension
concept in the closed case.

For instance, if a geometric chain has degree $d$ and it is
constructed from $\overline{\mathcal{M}}_{g,n}^{KSV}$, the real blow-up of the
DM-compactification of the moduli space as in \cite{HVZ}, we assign it
a new degree: $6g-6+2n-d$. In this new grading we also obtain a degree
one map. Indeed, if we have two chains of degrees $d_1$ and $d_2$
constructed from $\overline{\mathcal{M}}_{g_1,n_1}^{KSV}$ and $\overline{\mathcal{M}}_{g_2,n_2}^{KSV}$ respectively, their corresponding
 codimensions are
$6g_1 - 6 -2n_1 -d_1$ and $6g_2 - 6 -2n_2 -d_2$. After
 twist gluing we
 obtain a chain of degree $d_1 + d_2 + 1$ which
lives in $\overline{\mathcal{M}}_{g_1+g_2,n_1+n_2-2}^{KSV}$
 and therefore has codimension
\[ 6g_1 +6g_2 - 6 + 2n_1 + 2n_2 -4 -d_1-d_2-1 = 6g_1 + 6g_2 -2n_1 -2n_2 -d_1-d_2 -11. \]
 However, the sum of the original codimensions is
 $6g_1 + 6g_2 +2n_1 +2n_2-d_1-d_2 -12$ which shows that the
 change in degrees is exactly 1. In the self-twist
gluing picture something similar happens and the
 change in degree is 1 as well.

This grading by codimension may seem odd but it is
 exactly what we need in the open case. Recall that
the twist gluing appeared in the closed case because of
 the different choices one has to attach surfaces along
labeled points in the interior of the surface (different angles).
 This is not the case for labeled points in the boundary.

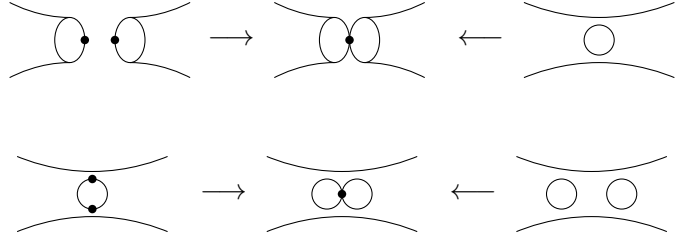
\begin{figure}
\begin{center}

\begin{tikzpicture}[scale=0.2]

\draw (-2,1.5) parabola (-6,2.5);
\draw (-2,-1.5) parabola (-6,-2.5);
\draw (2,1.5) parabola (6,2.5);
\draw (2,-1.5) parabola (6,-2.5);

\draw (-2,0) ellipse (1 and 1.5);
\draw (2,0) ellipse (1 and 1.5);

\fill (-1,0) circle (8pt);
\fill (1,0) circle (8pt);

\node[right] at (6.5,0) {$\longrightarrow$};

\end{tikzpicture}
\begin{tikzpicture}[scale=0.2]

\draw (-1,1.5) parabola (-5,2.5);
\draw (-1,-1.5) parabola (-5,-2.5);
\draw (1,1.5) parabola (5,2.5);
\draw (1,-1.5) parabola (5,-2.5);

\draw (-1,0) ellipse (1 and 1.5);
\draw (1,0) ellipse (1 and 1.5);

\fill (0,0) circle (8pt);

\node[right] at (6.5,0) {$\longleftarrow$};

\end{tikzpicture}
\begin{tikzpicture}[scale=0.2]

\draw (-5,2.5) parabola bend (0,1.5) (5,2.5);
\draw (-5,-2.5) parabola bend (0,-1.5) (5,-2.5);

\draw (0,0) circle (1cm);

\end{tikzpicture}
\par
\vspace{1cm}

\begin{tikzpicture}[scale=0.2]

\draw (-5,2.5) parabola bend (0,1.5) (5,2.5);
\draw (-5,-2.5) parabola bend (0,-1.5) (5,-2.5);

\draw (0,0) circle (1cm);

\fill (0,1) circle (8pt);
\fill (0,-1) circle (8pt);

\node[right] at (6.5,0) {$\longrightarrow$};

\end{tikzpicture}
\begin{tikzpicture}[scale=0.2]

\draw (-5,2.5) parabola bend (0,1.5) (5,2.5);
\draw (-5,-2.5) parabola bend (0,-1.5) (5,-2.5);

\draw (-1,0) circle (1cm);
\draw (1,0) circle (1cm);

\fill (0,0) circle (8pt);

\node[right] at (6.5,0) {$\longleftarrow$};

\end{tikzpicture}
\begin{tikzpicture}[scale=0.2]

\draw (-5,2.5) parabola bend (0,1.5) (5,2.5);
\draw (-5,-2.5) parabola bend (0,-1.5) (5,-2.5);

\draw (-2,0) circle (1cm);
\draw (2,0) circle (1cm);

\end{tikzpicture}

\end{center}
\caption{\label{openfig} Boundary degenerations for the open case.}
\end{figure}

If we consider surfaces with at least one marked point in all boundary components we have
 essentially two cases for the boundary degeneration shown in
Figure \ref{openfig}. In the first one we have two labeled points
in different boundary components and in the second we have two labeled
 points in the same boundary component. The surface on the center is
the result of attaching the labeled points represented on the left.
The surface on the right is the desingularized version of the one in
 the center. Since there are no ambiguities in how to attach the
labeled points this operation induces a degree zero map. However,
 grading by codimension is a completely different story. In the
 first case we have two chains of dimensions $d_1$ and $d_2$
 respectively. Recall that the dimension of the moduli space
 $\overline{\mathcal{M}}_{g,n}^{HVZ\, b,\vec{m}}$ is $6g - 6 + 2n + 3b + m$ where $b$ corresponds to the number of boundary
components and $m$ is the number of labeled points in this boundary as
 in \cite{HVZ}. The codimensions are
then $6g_1 - 6 + 2n_1 + 3b_1 + m_1 - d_1$
and $6g_2 - 6 + 2n_2 + 3b_2 + m_2 - d_2$ respectively and
their sum is \[ 6g_1 + 6g_2 +2n_1 + 2n_2 +3b_1 + 3b_2 + m_1 + m_2 - d_1 - d_2 -12. \] After attaching, the new chain
lives in $\overline{\mathcal{M}}_{g_1 + g_2,n_1 + n_2}^{HVZ\, b_1+b_2-1,\vec{m}'}$,
where $\vec{m}'$ has $m_1 + m_2 - 2$ components, and therefore its
codimension
is \[ 6g_1 + 6g_2 - 6 + 2n_1 + 2n_2 +3b_1 + 3b_2 -3 +m_1 + m_2 -2 -d_1 - d_2 \]
which is equal to
\[ 6g_1 + 6g_2 +2n_1 + 2n_2 +3b_1 + 3b_2 + m_1 + m_2 - d_1 - d_2 -11 \]
 and therefore we get a degree one map again.
Similar calculations take care of the self attaching operation and the
second case.

Geometrically, the grading reflects the chosen orientations. And it
is this choice of orientation \cite{HVZ} that makes the gluing
odd.

Intuitively, in the closed case there is an extra vector
being added in the tangent bundle due to the circle. But there is also
another vector being added in the normal bundle. In the open case
there is an additional vector being added only in the normal bundle
so grading by codimension gives us an odd gluing.

\subsubsection{Open/closed interaction; adding a derivation.}
 This idea is also the guide if we consider surfaces without marked points in some of their boundary components. In this case there is a new phenomenon that occurs in the boundary.
Namely,
 as a boundary component degenerates it actually turns into something that looks like a marked point (a puncture in fact). Therefore it is essential to consider a new operation that simply re-labels a marked point as a degenerate boundary component in order to balance the quantum master equation.

 If we make the same computation we did before for chains using codimension we also encounter a degree one map. However it is very clear in this case that we are not really changing the chain, we are just placing it in a different moduli space and hence changing the codimension. This is an interesting interaction between the closed and open operations and it is like twist gluing a surface at an interior (closed) marked point with a disc with only one interior marked point at such point giving a sort of degenerate boundary.

This open/closed interaction given by this degeneration leads to a contribution $\Delta_{oc}$
which is not only a derivation, but also a derivation of degree $1$.
Of course adding a degree $1$ derivation to a BV operator which anti-commutes with it
results in a new BV operator.

\subsubsection{Master Equations and Compactification}
In the above cases, we see that the fact that $S$ which is composed
out of fundamental classes,  satisfies the QME is equivalent to the fact
that the boundary divisors are either given by twist gluing two curves
$\cbullet{i}{j}$ or self--gluing the curves $\bullet_{ij}$ or the open gluing.

\begin{qu}
What is the meaning of the ME or QME in the context of the $\Arc$ operad,
the framed little discs and the Sullivan PROP?

There are two things which have to be solved (1) what kind of chains
(2) what is the correct notion of fundamental chains.

For $\Arc$ there is a partial compactification, while the Sullivan PROP
retracts to a CW complex, so one can use cellular chains. A clue might be provided by
the Stasheff polytopes and the $A_{\infty}$ Deligne conjecture \cite{KSchw},
see below \S\ref{Ainfpar}.
\end{qu}

It seems that a fundamental role for the $\Arc$ or Sullivan PROP is
played exactly by the arc families whose arcs do not quasi--fill the surface.
Recall that an arc family quasi--fills the surface if its complement
are finitely many polygons which contain at most one puncture, see \cite{hoch1,hoch2}.

\subsection{Other examples: $A_{\infty}$ and $A_{\infty}$ Deligne}
\label{Ainfpar}
The Stasheff polytopes are also a geometric incarnation of the master equation.
This follows e.g.\ from \cite{HVZ}, where discs with boundary points are used.
But even classically the boundary of an associahedron,  is precisely
given by all possible compositions of lower order associahedra.
This is precisely the compactification
 one would get for the non--sigma bracket
and the corresponding master equation.

 The link
to the algebraic world is then to take a chain model where the
usual power series of fundamental classes rel boundary gives
a solution to the ME.

This is taken a step further in \cite{KSchw} where a product of cyclohedra and
associahedra was given as the topological operad  lying above
the minimal operad of \cite{KS} which in our framework is a Feynman
transform of the Poisson operad $Assoc\circ Lie$.

\subsection{Topological Feynman transform?}
One question that remains is what is the general theory of a topological
Feynman transform.

For the closed type the set could be:

\begin{equation}F\O((g,S))=\bigsqcup_{\colim(\G(g,S)\downarrow \crl_{g,S})}
\bigsqcup_{v\in V_{\G}}\O(\crl_v)\bigsqcup_{e\in E_{\G}}S^1
\end{equation}

This could be considered as a real blow up of the DM compactification.
However, it is the way that this set is topologized which is not clear.

Furthermore there are the open examples, where the $S^1$ factors disappear
in favor of more structure at the vertices.
In all one could make the following tentative definition.

\begin{df}
A topological Feynman transform of a modular operad $\O$ is
a collection of spaces $\overline\O((g,n))$ with $\O((g,n))\subset\bar \O((g,n))$
such that there are fundamental
classes coming from the relative fundamental classes which
satisfy the quantum master equation.
\end{df}
Examples are then the moduli spaces above and the associahedra
as well as the topological model for the minimal operad of Kontsevich and Soibelman \cite{KS}.

This  is essentially equivalent to the cut--off view of Sullivan \cite{S1,S2,S3}.
Here the cut off is given by removing a tubular neighborhood of the compactification divisor which
amounts to a real blow--up of that divisor.

\begin{rmk}
Notice that in more involved cases, like the open/closed version, there might be
several terms in the master equation. Basically there is one term for each type of elementary operation. Closed self--, closed non--self--, open self--, open non--self--gluing and open/closed degeneration. This theme is explained
in \cite{feynman} where we define a Feynman transform relative to a set of generating morphisms.
\end{rmk}
\begin{rmk}
Considering the master equations from the chain level, the master equation here could be interpreted as
giving a morphism to  the trivial modular operad. This of course can be viewed as pushing forward to a point,
which is what integration is.
\end{rmk}

\section{Summary and Discussion} Let us conclude by putting our results in the context of the existing literature. First we recount our mantra along with the mathematical statements from the text which realize them:

\begin{enumerate}
\item Odd non-self-gluings give rise to odd Lie brackets: Corollary $\ref{r1}$, Theorem $\ref{r2}$, Theorem $\ref{r3}$.
\item Odd self-gluings give rise to differential operators:  Theorem $\ref{r4}$, Theorem $\ref{r5}$.
\item The horizontal multiplication   turns the odd brackets into odd
Poisson or Gerstenhaber brackets and makes the differentials BV operators: Theorem $\ref{r6}$, Theorem $\ref{r7}$, Theorem $\ref{ncthm}$.
\item Algebraically, the master equation classifies dg-algebras over the relevant dual or Feynman transform: Theorem $\ref{r8}$.
\item Topologically, the master equation drives the compactification:  Section $\ref{toppar}$.
\end{enumerate}

\subsection*{Mantra 1 and 2}  The definition of the bracket and the signed bracket for operads are classical and go back to \cite{Gerst,KM}.  The bracket in the general odd cyclic case is new. Several examples of such Lie structures have appeared in various guises in
\cite{Kthree, CV,BLB,Ginz,Schedler,Ba,MenLie}.
Many of these examples are given by anti-cyclic structures, that
 arise from tensoring particular cyclic operads with a particular
anti--cyclic operad given  by the endomorphism operad of symplectic
vector spaces \cite{Kthree,CV,Ginz}.  These are now all corollaries.

We also clarified that the bracket lifts to the cyclic coinvariants and
give the explicit relation to the non--cyclic bracket.

(Wheeled) PROP(erad)s  have been
extensively considered in \cite{KM,Val,MV,Merkulov,MMS}.
There are several differences to the theory of \cite{Val, MV}
though. The most important is that we
 only use only {\em single edge gluings} and
we {\em do not} include the horizontal composition for both the brackets and differentials, as well as for the dual transform.
 In particular, this bracket includes
only the dioperadic gluings. This means that the results  do not directly transfer, but have to be
adjusted and newly justified. For instance it is not {\it a priori} clear that the Lie--admissible for Properads structure restricts \cite{MV,Fiore}. The associative structure for PROPs \cite{KM} e.g. restricts to only Lie--admissible.
Furthermore, one has to watch out for different sign conventions in these cases. The $\K$--twisted version is however different from the vertex suspension of \cite{MV,Val}.
As a corollary of the general statement we can recover the operator $\Delta$ which was found in \cite{Merkulov} in the co--free case of the co--bar transform.

For $\K$--twisted modular operads the statement is new in this generality. A corollary is that this applies to the Feynman transform of a modular operad. This example was found in \cite{Ba}.

One main point we establish is that it suffices to have a $\K$ twisted structure. It is not necessary to be quasi--free, a convolution product, a tensor product involving symplectic $\Endo$--operads or any of the other special examples.

\subsection*{Mantra 3}
This theorem in its generality on the algebraic level is new. Topologically this goes back to \cite{Zwie}.
The nc-versions of operads are new, except the model given by the PROP generated by an operad. New as well are
the nc-generalization for cyclic and modular operads. For (wheeled) Properads of course the nc-version are  by definition
wheeled PROPs.  Here  our new point of view the horizontal composition is
not on equal footing and should not be included in the bracket, but rather gives a new multiplication.
Again for this one needs to consider the $\K$ twist and not the one by vertex suspension.
The fact that $\Delta$ becomes BV and the signed bracket
odd Poisson is then new. In \cite{KM} for instance it was part of an associative algebra structure.

The nc-version for modular operads, see \S\ref{ncmodularsec}, is new. It is related to the MOs of \cite{Schwarz} via taking coinvariants.

\subsection*{Mantra 4} For operads  this is classical. For properads it can alternatively be proven from by using Theorem 4.1.2. in \cite{Merkulov}
where now the quasi-free object is the dual transform involving {\em only one-edge  gluings.}  Again one has to use spurious shifts.
For $\K$-modular operads this result is contained in \cite{Ba}.
For $\K$-twisted cyclic operads it then follows by restriction, although this has not appeared in the literature.
The extension to the nc-cases and the identification of the terms as Gerstenhaber brackets and BV operators in general are new.

As our results show, Hom spaces between structures that differ by a $\fr{K}$-twist provide a source of examples.  An example pertaining to
Feynman transforms appears in \cite{Ba}. The convolution operads and properads, \cite{MMS,MV}
are also examples, because of the possibility to shift in the directed cases.
In \cite{MMS} another special case  of $\Delta$ is given which classifies Master functions from geometry, see Theorem 3.4.3 of \cite{MMS}. Here
the authors in our language consider a particular $\K$ twisted wheeled properad which is the dual transform (in our sense) of a special wheeled properad.

\subsection*{Mantra 5}
A similar construction to our balanced $S^1$ actions appeared in \cite{Schwarz} and we thank A.\ Schwarz for pointing this out to us.

Various connections between compactifications and master equations have been studied in \cite{Zwie,KSV,HVZ,Costello, Kate}.
The application to the $\Arc$ operad and hence string topology are new.
Looking more carefully how an acceptable action $S$ can be built out
of fundamental classes, one can
say that by reverse engineering:
\begin{quote}
{\it Topologically the Master equation drives the compactification.}
\end{quote}

These considerations and our treatment of signs
might be helpful for further endeavors in string topology, see e.g.\ \cite{Kate}.

\subsection*{Outlook}

In  \cite{feynman}, we  give a general, abstract, categorical setup where all of the above types
of ``operad--like'' structures are on equal footing as functors
from so--called Feynman categories. The selected ones are examples of Feynman categories
 of Feynman graphs.
The odd versions of the structures
are then obtained by using a universal twist called $\K$,
which makes edges odd, viz.\ have degree $1$.
Making the abstract concepts concrete  in the examples
most relevant for the ``practicing mathematician or physicist'',
one is led back to the concrete constructions and calculations we present here.
And, in fact, the theory of Feynman categories
was motivated by the calculations
of this paper.  There we also consider the dual transforms in full generality and prove that they are relatively co-fibrant
after establishing the correct model category framework.

\section*{Acknowledgments}
We would like to thank B.\ Browder,
  M.~Kontsevich, Yu.~I.~Manin, J.~McClure,
A.~Schwarz,  J.~Stasheff, D.~Sullivan,  and S.~Voronov
for enlightening discussions. We also
benefited from the interest and  discussions with  C.~Berger, D.\ Borisov,
B.\ Fresse, S.\ Merkulov and B.\ Vallette.

RK thankfully acknowledges support from NSF DMS-0805881 and the
 Humboldt Foundation.
He also thanks the
Institut des Hautes Etudes Scientifiques, the Max--Planck--Institute
for Mathematics in Bonn and the Institute for Advanced Study
for their support and the University of Hamburg for its hospitality.

This project was worked on substantially while RK was visiting the IAS.
While at the IAS RKs work was supported by the NSF under agreement
DMS--0635607.

BW would like to thank the Purdue Research Foundation for its support.

 Any opinions, findings and conclusions or
recommendations expressed in this
 material are those of the authors and do not necessarily
reflect the views of the National Science foundation.

%\renewcommand{\theequation}{A-\arabic{equation}}
%\renewcommand{\thesection}{A}
% redefine the command that creates the equation no.
%\setcounter{equation}{0}  % reset counter
%\setcounter{subsection}{0}

\appendix
\section{Graphs and Algebras}
\subsection{The category of abstract graphs}
An abstract graph $\Gamma$ is a quadruple $(V_{\Gamma},F_{\Gamma},$
$\imath_{\Gamma},\del_{\Gamma})$
of a finite set of vertices $V_{\Gamma}$ a finite
set of half edges or flags $F_{\Gamma}$
and involution on flags $\imath_{\Gamma}:F_{\Gamma}\to F_{\Gamma}; \imath_{\Gamma}^2=id$
 and
a map $\del_{\Gamma}:F_{\Gamma}\to V_{\Gamma}$.
We will omit the subscripts $\Gamma$ if no confusion arises.

Since the map $\imath$ is an involution, it has orbits of order one or two.
We will call the flags in an orbit of order one {\em tails}.
We will call an orbit of order two an {\em edge}. The flags of
an edge are its elements.  The set of vertices and edges form a 1--dim simplicial complex.
The realization of a graph is the realization of this simplicial complex.

\begin{ex}
A graph with one vertex is called a corolla.
Such a graph only has tails and no edges.
Any set $S$ gives rise to a corolla. Let $p$ be a one point set
then the corolla is $\crl_{p,S}=(p,S, id,\del)$ where $\del$ is the constant map.
\end{ex}

Given a vertex $v$ of $\Gamma$ we set
$F_v=F_v(\Gamma)=\del^{-1}(v)$ and call it {\em the
flags incident to $v$}. This set naturally gives rise to a corolla.
The {\em tails} at $v$ is the subset of tails of $F_v$.  As remarked above $F_{v}$ defines a corolla $\crl_{v}=\crl_{\{v\},F_v}$.

\begin{rmk}
The way things are set up, we are talking about finite sets, so
changing the sets even by bijections changes the graphs.
\end{rmk}

An $S$ labeling of a graph is a map from its tails to $S$.

An orientation for a graph $\Gamma$ is a map $F_{\Gamma}\to \{in,out\}$
such that the two flags of each edge are mapped to different values.
This allows one to speak about the ``in'' and the ``out''
edges, flags or tails at a vertex.

\begin{ex}
A tree is a contractible graph. It is rooted if it has a distinguished
vertex, called the root. A tree has an induced orientation with the out edges
being the ones pointing toward the root.
\end{ex}

As usual there are edge paths on a graph and the natural notion
of an oriented edge path. An edge path is a (oriented) cycle if it starts and
stops at the same vertex and all the edges are pairwise distinct.
An oriented cycle with pairwise distinct vertices is sometimes called a wheel.
A cycle of length one is a loop.

A {\em na\"ive morphism} of graphs $\psi:\Gamma\to \Gamma'$
is given by a pair of maps  $(\psi_F:F_{\Gamma}\to F_{\Gamma'},\psi_V:V_{\Gamma}\to V_{\Gamma'})$
compatible with the maps $i$ and $\del$ in the obvious fashion.
This notion is good to define subgraphs and automorphism.

It turns out that this data is not enough to capture all the needed aspects
for composing along graphs.
For instance it is not possible to contract edges with such a map or graft
two flags into one edge. The basic operations of composition in an operad
viewed in graphs is however exactly grafting two flags and then contracting.
There is a more sophisticated version of maps given in \cite{BM} which we will use in the sequel \cite{feynman}. For now we
wish to add the following morphisms.

{\em Grafting.} Given two graphs $\Gamma$ and $\Gamma'$, a tail $s$ of $\Gamma$
and a tail $t$ of $\Gamma'$ then $\Gamma\scirct \Gamma'$ is the graph
with the same vertices, flags, $\del$, but where $\imath(s)=t$, and the rest
of $i$ is unchanged.

The {\em contraction of an edge $e$} of $\Gamma$ is the graph where
the two flags of $e$ are omitted from the set of flags and the vertices
of $e$ are identified. It is denoted by $\Gamma/e$.

{\em Merger.} Given two graphs $\Gamma$ and $\Gamma'$ merging
the vertex $v$ of $\Gamma$ with the vertex $v'$ of $\Gamma'$ means
that these two vertices are identified and the rest of the structures
just descend.

\begin{rmk}
One thing that is not so obvious is how $S$-labeling behaves
under these operations. If $S$ are arbitrary sets (the unbiased case)
this is clear. If one uses enumerations however (the biased case), one must specify how to re--enumerate. This is usually built into the definition of the
composition type gadget.
\end{rmk}

\subsection{Standard algebras}
For the readers' convenience,
we list the definitions of the algebras we talk about. Let $A$ be a graded
vector space over $k$ and let $|a|$ be the degree of an element $a$.
Let's fix char $k=0$ or at least $\neq 2$.

\begin{enumerate}
\item Pre--Lie algebra\index{algebra!pre-Lie}. $(A,\circ:A\times A \to A)$ s.t.
$$
a\circ (b\circ c)- (a\circ b)\circ c= (-1)^{|c||b|} [a\circ (c\circ b)-
(a\circ c)\circ b]
$$

\item Odd Lie. $(A,\{ \,\bullet \, \}:A\otimes A\to A)$ s.t.
\begin{enumerate}
\item odd anti-symmetry: $ \ \ \{a \bullet b\}= - (-1)^{(|a|-1)(|b|-1)}\{b \bullet a\}$
\item odd Jacobi:
$$0=\{a\bullet \{b\bullet c\}\}+(-1)^{(|c|-1)((|a|-1)+(|b|-1))}\{c\bullet \{a\bullet b\}\}
+(-1)^{(|a|-1)((|b|-1)+(|c|-1))}\{b\bullet \{c\bullet a\}\}$$
\end{enumerate}

\item Odd Poisson or Gerstenhaber\index{algebra!Gerstenhaber}. $(A,\{ \,\bullet \, \},\cdot)$
is odd Lie plus another associative multiplication
for which the bracket is a derivation with the appropriate signs. (Gerstenhaber is often also defined to be super-commutative.)
\begin{equation}
\label{dereq}
\{ a\bullet  bc\} = \{ a \bullet  b\} \ast c +(-1)^{(|a|-1)|b|} b\ast \{ a\bullet c\} \ \forall \ a,b,c   \in  A
\end{equation}

\item (dg)BV\index{algebra!BV}. $(A,\cdot, \Delta)$. $(A,\cdot)$
associative (differential graded) supercommutative algebra, $\Delta$ a differential of degree $1$:
$\Delta^2=0$ and
\begin{equation}\label{poissoneq}
 \{a\bullet b\}:=(-1)^{|a|}\Delta(ab)-a\Delta(b)-(-1)^{|a|}\Delta(a)b
\end{equation}
 is a Gerstenhaber bracket.

An equivalent condition for a  BV operator is
\begin{eqnarray*}
\Delta(abc)&=&\Delta(abc) \Delta(ab)c+(-1)^{|a|}a\Delta(bc)+(-1)^{(|a|-1)|b|}b
\Delta(ac)-\Delta(a)bc \nn\\ && -(-1)^{|a|}a\Delta(b)c-(-1)^{|a|+|b|}ab\Delta(c)
\end{eqnarray*}

\item (dg)GBV. This name is used if {\it a priori} there is a BV operator and a given Gerstenhaber
bracket and {\it a posteriori} the given Gerstenhaber bracket coincides with the one induced by the BV operator.
\end{enumerate}

% !TEX root = oddCMP.tex

%\bibliography{oddbib}
%\bibliographystyle{alpha}
\end{document}